\documentclass{lmcs}
\pdfoutput=1

\usepackage[utf8]{inputenc}

\usepackage{lastpage}
\lmcsdoi{17}{3}{6}
\lmcsheading{}{\pageref{LastPage}}{}{}%
{Jul.~16,~2020}{Jul.~21,~2021}{}

\usepackage{amsmath,amssymb,latexsym,textcomp}
\usepackage{graphics}
\usepackage{lipsum}

\numberwithin{equation}{section}

\def\uh{\upharpoonright}

\begin{document}
\title[Foundations of online structure theory II: The operator
approach]{\texorpdfstring{Foundations of online structure theory II:\\The operator approach}{Foundations of online structure theory II: The operator approach}}

\author{Rod Downey\rsuper{a}}
\address{\lsuper{a}Victoria University of Wellington}
\email{Rod.Downey@msor.vuw.ac.nz}

\author{Alexander Melnikov\rsuper{b}}
\address{\lsuper{b}Massey University}
\email{alexander.g.melnikov@gmail.com}

\author{Keng Meng Ng\rsuper{c}}
\address{\lsuper{c}Nanyang Technological University}
\email{kmng@ntu.edu.sg}

\thanks{The first two authors were partially supported by Marsden Fund of New Zealand. The second author was also partially supported by Rutherford Discovery Fellowship.  The third author is partially supported by the grant MOE2015-T2-2-055 and RG131/17.}

\begin{abstract}
We introduce a framework for online structure theory. Our approach generalises notions arising independently in several areas of computability theory and complexity theory. We suggest a unifying approach using operators where we allow the input to be a countable object of an arbitrary complexity.
We give a new framework which (i) ties online algorithms
with computable analysis, (ii) shows how to use modifications of notions from
computable analysis, such as Weihrauch reducibility,
 to analyse \emph{finite} but uniform
 combinatorics, (iii) show how to finitize reverse
mathematics to suggest a fine structure of finite analogs of
infinite combinatorial problems, and (iv) see how similar ideas
can be amalgamated from areas such as EX-learning, computable analysis,
distributed computing and the like. One of the key ideas is that online algorithms
can be viewed as a sub-area of computable analysis.
Conversely, we also get an enrichment of computable analysis
from classical online algorithms.
\end{abstract}

\maketitle


\section{Introduction}
\label{intro}

\subsection{Our Goal}
Imagine you are tasked with putting objects of differing sizes into bins of a fixed size. Your goal is to minimize the number of bins you need.
This is the famous {\sc Bin Packing} problem  which we know is NP
complete (see Karp \cite{Karp}). But imagine that we change the rules and
you are only given the objects one at a time and you must choose which bin to put the object into before being given
the next object. You are in an \emph{online} situation and this is the
{\sc Online Bin Packing} problem. The ``first fit'' method is
well-known to give a 2-approximation algorithm for this problem (definitions
given in detail below).  Alternatively imagine you are a scheduler, and your goal is to
schedule requests within a computer for memory allocation amongst users. Again
you are in an online situation, but here you might want to change the order of
allocation depending on priorities of the requests.
Or from algorithmic randomness, you have a (computable) KC-set of
requests of the form $(2^{-n_i},\sigma_i)$ with
$\sum_{i=1}^\infty 2^{-n_i}\le 1$, and need to build a prefix-free
Turing machine
$M$ with strings $\tau_i$ such that $|\tau_i|=n_i$ and $M(\tau_i)=\sigma_i$.
Then the proof from e.g. Downey and Hirschfeldt
\cite{DH}, Theorem 3.6.1 is online in the sense that for each request at step
$i$ we generate the string $\tau_i$.

Thus an online algorithm is one which acts on a input which is given piece by
piece in  a serial fashion. In the case where the input is finite, Karp
\cite{Karp2} suggested this as a sequence of ``requests'' $r_1,r_2,\dots$
with the algorithm $f$ specifying an action $f(r_1), f(r_1r_2),\dots$.
The natural model for this would be a database where a request would be
an update.
Note that this is quite distinct from the \emph{offline} version
where (in the finite case) the whole input is known in advance.
It is important to realize that in practical online algorithms
arising in computer science, the action needs to be specified
before the next request is given. Occasionally this is
varied with a \emph{lookahead} or \emph{delay} where
typically we might get $k$ further bits of  input
so $r_1,\dots,r_{n+k}$ determines the next action\footnote{As we will see below, it is possible to generalize this further
perhaps to have delay or lookahead $g(n)$ for action $n$
with $g$ some function of $n$ (or even $r_n$) but this is
not what happens in practice.
Most natural examples work from $r_n$ to define $f(r_n)$ and
indeed if $r_n$ is a structure generated by $\{1,\dots, n\}$ of some kind,
then $f(r_n)$ is some value $h(n)$.}.


A brief thought on this will reveal that there are potentially hundreds of situations where we
are dealing with combinatorial algorithms for tasks where we only have partial
evolving information about the input data, or perhaps the data is so
large that we cannot see it in total. This is the reason that there are so
many algorithms for online tasks.
Classical examples include insertion sort, perceptron,
paging, job shop scheduling, ski rental,
navigation with only local understanding,
etc.; see \cite{Albers}.
We will  discuss other possible approaches and the relevant literature in due course.
At this point we only note that books in this area, such as
Albers
\cite{Albers}, all tend to be taxonomies of algorithms.
Our goal is to give a theoretical basis for the theory of online
algorithms and structures which relies on uniform operator approach from computable analysis
and the classical notion of primitive recursion.


\subsection{The Punctual Model}
\label{punctualll}
In \cite{bsl} a related
 project was started aiming at providing a model-theoretical foundation
to this theory. In that paper we focussed upon the intuition that online
decisions in practice have \emph{lack of delay.} That is,
we need to pack the object into some bin immediately, before the next one
is presented to us (in the {\sc Bin Packing} example).
This led to a theory of online structures and algorithms
we generally referred to as \emph{punctual} structure theory.

What is the most general reasonable form of ``punctual''?
In \cite{bsl}  we gave several pages of analysis as to why
there we chose to interpret punctuality
 as \emph{primitive recursive}. That is,
we chose primitive recursive as a \emph{unifying abstraction} of
the notion of lack of delay.

To keep this paper self-contained we will repeat the arguments of \cite{bsl} here, so the reader familiar with \cite{bsl} might choose
to move on to Section~\ref{oper}.

\subsection{Turing computable mathematics}
The general area of \emph{computable} or \emph{effective mathematics} is devoted to
understanding the algorithmic content of mathematics. The roots of the
subject go back to the introduction of non-computable methods into
mathematics at the beginning of the 20th century as discussed in
Metakides and Nerode \cite{MNroots}. Early work concentrated on developing
algorithmic mathematics in algebra, e.g. Grete Hermann \cite{Herm1926},
analysis such as Bishop's constructive analysis,
(implicitly) using algorithmic methods to understand randomness
(Borel \cite{Bo}, von Mises \cite{vM}, Church \cite{Ch}),
understanding effective procedures in finitely presented groups
such as Dehn \cite{De}, and most notably Hilbert's programme
seeking to give a decision procedure for first order logic.
We know all of these historical roots led to the development of,
for example, computability theory,
complexity theory, and algorithmic randomness (see e.g.~Downey \cite{DoT}).
The modern version of effective mathematics utilizes the tools developed
in these areas, as well as classical tools in algebra, analysis, etc.~to calibrate the algorithmic content of many areas of mathematics.

The standard model for such investigations is a (Turing) computable presentation
of a structure.
By this we mean a coding of the structure with universe ${\mathbb N}$,
and the relations and functions coded Turing computably.
For example, a computable presentation of  a group would be
either a finite group, or one where the universe was considered as
${\mathbb N}$ and the group operation was represented as a (Turing) computable function.
Note that this framework uses the general notion of a Turing computable function. In particular,
we put no resource bound on our computation.

\subsection{Online  combinatorics}
A hallmark of the majority of algorithms on finite structures
is that the algorithm ``knows all about the structure''.
In other words, the whole structure is given to the algorithm at once. For example,
when a complexity theorist talks about the  Hamiltonian path problem, they have in mind algorithms
that given a description of a finite graph (say, a matrix-presentation of it) outputs such a path.
This is sometimes not
true for large data sets, and several {\sc Logspace} algorithms, but
we are using this to refer to those,  students would learn
in a basic algorithm course. What happens to such algorithms if the graph is not given at once, but rather
is given to us step-by-step and vertex-by-vertex? This situation is an abstraction to an ``online" computation in which  the input data is  too massive to be given as an input at once.
Now, as seen in the introduction, there are \emph{many} problems in computer science where
we can safely assume that the universe is infinite  and thus we need an online algorithm. For example, a
scheduler which assigns users to access shared memory
is a classic example.

 In the ``online" setup the situation becomes quite
a bit harder. Consider the following example.
Every tree is 2-colourable, but to achieve this colouring
you need to know the whole of the tree. Suppose we are given a vast tree
 one vertex at a time, so that $G=\cup_s G_s$,
an \emph{online} presentation of $G$.
When we give you the vertex $v$ we promise to tell you all of the
vertices given so far to which $v$ is joined; that is the induced subtree
of $v_1,\dots,v_s$.
Your goal is to colour the vertex $v_s$, before we give you
$v_{s+1}$.
We  are in an \emph{online}
situation. For  a tree
with $n$ vertices,
the sharp lower bound is  $O(\log n)$ many colours.
It follows that there are
online  presentations of infinite (computable) trees which
cannot be online coloured with any
finite number of colours.
We see that switching to the online case affects not only the running time, but also the
best solution that we can hope for.
We remark that online algorithms can be quite complex. 

Beginning in the 1980's
there has been quite a lot of work on online infinite combinatorics, particularly
by Kierstead, Trotter, Remmel and others (\cite{KieTrans,hand,KT,LST,Remmel86}). Some results
were quite surprising. For example, Dilworth's theorem says that a
partial ordering of width $k$ can be decomposed into
$k$ chains. Szemeredi and others showed that there is a
computable partial ordering of width $k$ that cannot be decomposed into
$k$
computable chains. But in 1981, Kierstead proved that
there is an online algorithm that will decompose
any online presentation of a computable partial ordering
into $\frac{5^k-1}{4}$ many (computable) chains.
Only in 2009 was this result improved
by Bosek and Krawczyk who demonstrated that it can be done with
$k^{14\log k}$ many chains.
In the case of finite structures most work
comes from  comparing
offline vs online performance.
In this area, the typical setting is  to
build some kind of function which is measured
relative to some size, and the goal of online
algorithm design is to improve what is called the
\emph{Competitive Performance Ratio} of online divided by offline. For example, first fit
gives a competitive ratio of 2 for the classical
{\sc Bin Packing} problem (see Garey and Johnson \cite{GJ}).

\subsection{Online vs.~Turing computable}

The notion of an ``online'' algorithm in the results mentioned above is rather specific. One may complain that, rather than saying that we must make a decision before the next vertex shows up, it is fine to wait for a bit more of a graph to be shown to us. But \emph{how much more} exactly?  Maybe we can wait for $17$ more vertices to show up before we make a decision. Perhaps, at stage $s$ we could ask for  $\log(s)$ more vertices, etc. It is not hard to see that various answers to this question will lead to a proper hierarchy --  rather, a zoo -- of ``online'' computability notions.
It is natural to ask:
\begin{center}
\emph{What is the most general notion of an online algorithm?}
\end{center}

Understanding  the \emph{online content} of mathematics so far
has no general theory, there are only algorithms or proofs that
no algorithm  exists. Note that the lack of theory for online mathematics
stands in stark contrast with the infinite off-line case described by the  computable structure theory~\cite{AshKn, ErGon}.
However, as we noted above, computable structure theory relies on
the most general notion of a computable process that we know today --  a Turing computable process. Turing computability provides us with many tools, such as the universal Turing machine and the Recursion Theorem, that are useful in \emph{proving theorems about algorithms}. However, Turing computability in its full generality is not an adequate model in the online situation, because Turing computable algorithms can use an unbounded search. For instance, recall the example in which we had to online  colour a tree.
A Turing computable algorithm would just \emph{wait} until a node gets connected to the root of the tree via a path and then will make a decision.
There is no  \emph{a  priori}
bound on how long it may take for the path to be revealed, but a Turing computable algorithm does not care.
More importantly, Turing computability fails to capture the ``impatient'' nature of an online algorithm which has to make a decision ``now''.

\subsection{Our goal, revisited}
Recall that our goal is to give a general abstract foundation
for online algorithms.
As we will soon see,
our approach is based on  one natural interpretation of ``online'' involving
\emph{primitive recursive} structures.

In \cite{bsl}, using some of  the techniques and intuition coming from the mentioned above (Turing) computable structure theory~\cite{AshKn,ErGon}
we developed a theory contrasting and comparing classical computable structure theory
with an online ``punctual'' framework.
In \cite{bsl}, we discussed the following models.





\subsection{The  models} We will concentrate on infinite structures. Still to do is
to develop an appropriate  model theory for
online finite structures as asked for by Downey and
McCartin \cite{DM}.
In Section~9 of \cite{bsl}, we foreshadowed the developments of the present paper
which works some way towards addressing finite online model theory.

In its most general formulation, an online algorithm would act
on a structure ${\mathcal A}$ given in stages $f(1), f(2), \dots$,
where $f$ is a computable function representing timestamps.
At stage $f(n)$ we would enumerate $n$ into the partial structure $A_{f(n)}$
and give complete information about how $n$ relates to $\{0,\dots,n-1\}$.

Now the question is: \emph{What kinds of structures and time functions should
be allowed?} Different choices will result in different theories.
Our goal is to give a general setting that also reflects the
common online structures encountered.
We examine some approaches from the literature:





\subsubsection{Automatic structures}
Khoussainov and Nerode~\cite{KhoussainovNerode95}
initiated a systematic study into automatically presentable algebraic
structures; but these seem quite rare. For example, the additive group of the rationals is not automatic~\cite{tsankov}. The approach via finite automata is highly sensitive to
how we define what we mean by automatic. For example treating a
function as a relation yields quite a different kind of
automatic presentation. See \cite{WordProcessing} for an alternate approach to automatic groups.
Although the theory of automatic structures is a beautiful subject, a finite automaton is definitely not a general enough model for an online algorithm.

\subsubsection{Polynomial time computable structures}
 Cenzer and Remmel, Grigorieff, Ala\-ev,
and others
\cite{CR-survey98,Grigorieff90,ALA,Alaev-II} studied
polynomial time presentable structures. We omit the formal definitions, but we note that they are sensitive to
how exactly we code the domain. In many common algebraic classes we can show that all Turing computable structures have  polynomial-time computable copies.
 One attractive result is
that every computably presentable linear ordering
has a copy in linear time and logarithmic space~\cite{Grigorieff90}.
Similar results hold for  broad subclasses of Boolean algebras~\cite{cerem}, some commutative groups~\cite{ceremab,ceremdo}, and some other structures~\cite{cerem}.

\subsubsection{Fully primitive recursive structures}
As was noted in \cite{KMN2}, many known proofs from
polynomial time structure theory
 (e.g., \cite{cerem,ceremab,ceremdo,Grigorieff90})  are  focused on  making the operations and relations on the structure \emph{primitive recursive}, and then observing that the presentation that we obtain is in fact polynomial-time.

The restricted Church-Turing thesis for primitive recursive functions says that a function is primitive recursive iff  it can be described by an algorithm that uses only bounded loops. For example, we need to eliminate all instances of  WHILE $\ldots$ DO,  REPEAT $\ldots$ UNTIL, and GOTO in a PASCAL-like language.



It is not difficult to construct an example of a structure which is primitive recursive but does not have a polynomial-time presentation; see the introduction of \cite{cerem} for one such example.
  Nonetheless,  primitive recursion plays a rather important intermediate role in transforming (Turing) computable structures into polynomial-time structures.
Furthermore,
to illustrate that a structure has no polynomial time copy, it is sometimes easiest to argue that it does not even have a copy with primitive recursive operations, see e.g.~\cite{ceremab}.
In \cite{bsl} the intuition above
led us to systematically investigate into those  structures that admit a presentation with primitive recursive operations, as defined below. Kalimullin, Melnikov, and Ng~\cite{KMN2} proposed that an ``online'' structure must minimally satisfy:

  \begin{defiC}[\cite{KMN2}]\label{maindef}
   A countable structure is \emph{fully primitive recursive} (fpr) if its domain is $\mathbb{N}$ and the operations and predicates of the structure are (uniformly) primitive recursive.
 \end{defiC}
\noindent   The main intuition is that we need to  define more of the structure  ``without delay''.
Here ``delay'' really means an instance of a truly unbounded search.  We informally call fpr structures  \emph{punctually computable}. We could also agree that all finite structures are also punctual by allowing initial segments of $\mathbb{N}$ to serve as their domains\footnote{Although the definition above is not restricted to finite languages, we will never consider infinite languages in the paper.}.

\begin{rem}
The word ``fully'' in ``fully primitive recursive'' emphasises that the domain must be the whole of $\mathbb{N}$ and not merely a primitive recursive subset of $\mathbb{N}$; these are provably non-equivalent assumptions. If the domain could be merely a primitive recursive subset of $\mathbb{N}$ then we can  delay elements from appearing in the structure; this way one can easily  show that each Turing computable graph has a primitive recursive copy (\cite{bsl}). We decided that structures in which elements can be delayed are not really online.
 \end{rem}

Our goal is to give a most general setting that also reflects the
common online structures encountered.
From a logician's point of view, where do computable structures come from?
One of the \emph{fundamental results} of computable structure theory
is that:

\centerline{\emph{A decidable theory has a decidable model\footnote{Recall that a complete, first-order theory in a computable languadge is decidable if the collection of  all G$\rm\ddot{o}$del codes of its sentences forms a computable set. A model  upon the domain $\mathbb{N}$ is decidable if there is a Turing computable algorithm which, given a first-order formula with parameters from the domain of the structure, can decide whether the formula holds in the structure.}}.}

The proof of this
elementary fact is to observe that the Henkin construction
is effective, in that if the theory is decidable then
the constructed model is decidable as a model.
Many standard computable structures come from decidable theories.

Most natural decidable theories  are elementary decidable in
that the decision procedures are relatively low level. We have to go out of
our way to have natural decidable theories
whose decision procedures are not primitive recursive.
In \cite{bsl} we observed  that a theory with a
primitive recursive decision procedure has a
model which is decidable in a primitive recursive sense.

\subsubsection{The upshot}
Thus in \cite{bsl}, we chose fully primitive recursive structures
as our central model. Primitive recursiveness gives a useful
\emph{unifying abstraction} to
computational processes for structures with computationally bounded
presentations. In such investigations we only care that there is \emph{some} bound.
Furthermore, these models arise quite naturally through
standard decision procedures.

In \cite{bsl}, we also noted that many results we stated in terms of primitive recursion, can likely be pushed to polynomial time
structures. Furthermore, some of our counterexamples can in fact be stated in terms of \emph{any} class with sufficiently nice closure properties; e.g., for a class of total computable functions having a uniformly computable enumeration and closed under composition and primitive recursion.
 However, this does not mean that our choice of primitive recursive algorithms as a central model
is fairly arbitrary.  The mentioned above generalisation to a class of total functions can be viewed as  a version of the \emph{subrecursive relativisation} of primitive recursion.
The study of relativised versions of our results is interesting on its own right, but it is not really beyond the primitive recursive model.
Kalimullin, Melnikov and Montalb{\' a}n (in progress) have recently announced a number of unexpected results connecting relativised primitive recursive presentations
with syntax in the spirit of Ash and Knight~\cite{AshKn}. Also, as we see in the present paper, an expert in computable structure theory would know that relativisation is tightly connected with uniformity. Generalisations to polynomial time classes seem to require significant effort in some instances.
 Alaev~\cite{ALA, Alaev-II} has recently initiated a research program focused on extending these ideas to polynomial time algebra. Dealing with polynomial time algorithms requires specific techniques and counting combinatorics; this is something we do not have to worry in our more ``relaxed'' model.
 In contrast with, e.g., automatic algorithms or polynomial-time algorithms,  there is a highly convenient and clear version of  Turing-Church thesis for primitive recursive functions (see above). We will use the thesis throughout the article without explicit reference. It will allow to simplify our proofs and proof sketches.
Irrelevant counting combinatorics is stripped from such proofs, thus emphasising the effects related to
the existence of a bound in principle (rather than specifying the bound). These effects are far more significant than it may seem at first glance.

Models such as automata-based
structure theory~\cite{KhoussainovNerode95, Hod, Hod1},
are highly sensitive to
presentations of the structures. For example treating the algorithms as
generated by transducers yields a completely different theory to that
obtained by treating functions as relations, as can be seen by comparing the approach
of Khoussainov and Nerode~\cite{KhoussainovNerode95}, with that of Epstein et al.~\cite{WordProcessing}.
Also it would seem that although we can incorporate
automatic  processes in our theories, they are really not general
enough for online algorithms in general. Similarly, polynomial time structures
such as Cenzer and Remmel, Grigorieff, Ala\-ev,
and others
\cite{ALA,Alaev-II,ASel,CR-survey98,Grigorieff90} are rather presentation dependent.
Finally, primitive recursive has a nice Church-Turing thesis,
in that it models computable processes without unbounded loops.

\section{The Uniform/Operator  Model}
\label{oper}

Whilst the \cite{bsl} model is a natural model, as we observed in Section~9
of that paper, there are aspects of online combinatorics which are
\emph{not} covered by the model.

Imagine we need to build a colouring
of a graph $G$ which is given online. Thus, in the very simplest case, we would
be given the graph $G=\lim_s G_s$, where $G_s$ has $s$ vertices. When
the vertex $s$ is introduced, we are also given at the same time
precisely which vertices amongst $\{1,\dots,s-1\}$ has an edge with $s$ (and this cannot change later).
(This is the ``request set'' in Karp's paper.)
Our task is to colour $s$  so that no two vertices which are
connected have the same colour, before the opponent presents us with $G_{s+1}.$

Although in practice the task will be finite, since we have no idea how large
the graph is, we can construe this as an infinite process.
Imagine this online colouring of a finite graph 
as an infinite process where  we need to colour the
whole of an infinite graph $G$ given to us as incremental induced subgraphs. We can think of each possible
version of $G$ as being a path through an infinite tree
of possibilities. Each node $\sigma$ of length $s$ of the tree will represent some graph
$G_\sigma$ with $s$ vertices, and
if $\sigma\prec \sigma'$ then $G_\sigma$ is the subgraph of $G_{\sigma'}$
induced by vertices $\{1,\dots,s\}$.
Note that there are only primitively recursively many non-isomorphic graphs with $s$ vertices\footnote{Meaning that this number is  $u(s)$, where $u$ is primitive recursive.}.

Then this view of an online algorithm differs
from that given in \cite{bsl} for the following core reason:

\begin{quote}
 Although $G$ can be viewed a path on an infinite
primitive recursive tree of possibilities, there is no \emph{a priori} reason that we should only consider a primitive recursive graph $G$. There are continuum many such paths and the online graph colouring problem can be considered for an infinite countable graph of any complexity.
\end{quote}

The reader will quickly realise that the key point about online algorithms
is one of \emph{continuity} or \emph{uniformity}. If we have a colouring
of $G_\sigma$ and we add a new vertex $s+1$, the next $G$ will be one of the
possible extensions $G_\tau$  of $G_\sigma$ with the vertex $s+1$ added. For each such $G_\tau$ the colouring $\chi_{G_\tau}$ must be compatible with the colouring $\chi_{G_\sigma}$ on $G_\sigma$. 

\paragraph*{Computable Analysis.}
The conclusion is that whilst online algorithms appear to be
combinatorial algorithms on finite objects and possibly infinine ones, in fact they should be
formulated as a branch of \emph{computable analysis}.
One of the goals of the present paper is to give such a formulation.
We need to specify what kinds of spaces are of relevance and  what kinds
of operators correspond to online algorithms.
We believe that this view will allow a discourse between the discrete and
the continuous which could prove fruitful. Similar relationships between
the continuous
and the discrete have yielded powerful results such as
in the Furstenberg view of Szemeredi's Theorem. Our unifying
abstraction also means that
computable analysis is shown to be important in finite combinatorics
(see also Avigad~\cite{avi}).
We will also see that our abstraction means that we can relate
the proof theory of finite combinatorics with classical
proof theory, obtaining refinements of result for Reverse Mathematics, and
we can relate the theory of incremental computation
with ideas from computable analysis such as Weihrauch reducibility.
Thus although this is not the most technically difficult paper,
we see it as a conceptual advance showing that many ideas
can be combined into a single unifying abstraction.

\paragraph*{Immediate actions.}
Since we want the action to be immediate, following the
abstraction of \cite{bsl} and for the reasons above, we will also demand that the action
works primitive recursively, or perhaps even running in polynomial time.
Thus, for the example above, the most general online colouring algorithm must satisfy the following two features:

\begin{itemize}

 \item[-] $\chi_{G_\tau}$ must agree with  $\chi_{G_\sigma}$ for $\sigma \prec  \tau$.
 \smallskip
 \item[-] The map $\tau \rightarrow \chi_{G_\tau}$ must be (minimally) \emph{primitive recursive}.
\end{itemize}

These will all soon be made precise and general using representations (i.e, naming systems) for online problems.
In fact, there are at least three possible interpretations of the second clause above. For example, should we allow \emph{lookahead} or \emph{delay} in the computation of $f$? In Lemma~\ref{lem:ptime}
we will prove that all three potential definitions are equivalent up to a primitive recursive change of notation, and therefore our definition is \emph{robust}.

\subsection{Why primitive recursion?} Should we instead require  the solution to  be polynomial time?
As we have already mentioned above, this notion would be too notationally dependent to be unambiguous. Also, it is well-known that in practice some useful algorithms are (provably) not polynomial-time; yet they seem to perform well enough on most inputs.
It is therefore not even clear if polynomial-time is the right abstraction for efficiency in the online situation.
So  we want our function to be in a nice complexity class but we are not  yet sure what exactly this class should be. It makes sense to develop as much structural ``punctual'' theory as possible and then see how much of it is preserved when we restrict ourselves to some narrow complexity class. And if something fails, we will have a better idea what goes wrong in the worst possible scenaria; e.g., we will compare the positive Theorem~\ref{thm:anal} and the analogous negative result in polynomial-time analysis~\cite{Ko}. Perhaps, it will help to define generic-case online algorithms in the future.

 As with classical
complexity theory, there is usually a \emph{natural} representation for a problem
we are interested in.
 The reader will note that in our definitions below,
the actual representation \emph{does} affect what we will regard as online.
Nonetheless,
one of the main advantages of our rather general primitive recursive approach is that we still can prove
a number of notation-independent results; e.g., the ``robustness" Lemma~\ref{lem:ptime} and the above-mentioned Theorem~\ref{thm:anal}.
Such results focus more on the effects related to online-ness and less
on the pathologies related to a specific choice of representation. Analogous results usually fail if we restrict ourselves to, say, polynomial-time algorithms because passing from one representation to another
can be computationally too hard.

On the other hand, we will also prove several results (e.g., Theorem~\ref{thm:or}) which show that sometimes
\emph{all} pathologies  come from presentation because the only notation-independent online solutions are the trivial ones.
Such results of the second kind will typically hold for polynomial-time or exponential (etc.) algorithms too, and via essentially the same proof. Primitive recursion serves here as a unifying abstraction rather than an idealisation.

We will also see that, modulo subrecursive relativisation, the earlier approach to online algorithms by  Kiersetead, Trotter et al.~\cite{hand,LST,KT} and Borodin and El-Yaniv~\cite{onlinebook} can be viewed as a special case of our framework.
According to this earlier approach, the map $\tau \rightarrow \chi_{G_\tau}$ just needs to be total and does not even have to be computable. So we see that primitive recursion is not that general when compared to some other definitions in the literature.

 Should we perhaps use  (a use-restricted form of)
general Turing computability  in place of primitive recursion?
It is more general, and some analogy of the above-mentioned robustness lemma (Lemma~\ref{lem:ptime}) will still hold.
The key difference here is that it would hold for a completely different reason.
A recursion theorist will be well-aware of how much unbounded search is abused in many such proofs.
For example, we can use compactness of the representation space and \emph{wait} for it to be covered by open sets.
There will perhaps be no bound on how long we will have to wait, but from the point of view of Turing computability it will not make any difference.

However, primitive recursion seems just general enough for many structural results to hold, but often via a different, more subtle argument which takes into account punctuality of our procedure. A fine example of such a theorem will be given in Section~\ref{real}
where we prove an online version of a well-known theorem of Weierstrass. It is very easy to show using a compactness argument that the theorem holds (Turing) computably. But it requires some thought and a completely different argument to see why it holds punctually.

As mentioned above, we realise that this material also has a connection
with \emph{computable} and \emph{feasible analysis}, and also with the
complexity theory for operators in analysis along the lines of
Kawamura and Cook \cite{KC}, Melhorn \cite{Mel},  Ko and Friedman \cite{FK},
and others.
We will also note  connections with reverse mathematics, computational learning theory, and even algorithmic randomness.
We will also see that, in this setting, the finiteness of the objects being given is not an essential restriction. In the online case, finite objects are only revealed one bit at a time, and for all intents and purposes, we may as well treat all inputs as arbitrarily large finite structures. We will prove that under the uniform operator framework, working with arbitrarily large finite structures and
infinite structures are indeed the same for our setting. This allows for example, for a formal approach in which one can study finite combinatorics in reverse mathematics.

\noindent\begin{minipage}{\linewidth}
As a final remark,
we mention that we see this work as an extension of \cite{bsl} in the following
way.

\begin{quote}
\cite{bsl} considered online computation of primitive recursive structures,
with primitive recursive functions. This is akin to
The Turing-Markov \cite{Tur}
view of computable analysis  as effective
processes on the countable field of
\emph{computable} reals. The Gregorczyk-Kleene
 \cite{Gr55}
views, called
type II computability,
which views computable analysis as effective operators acting
on the continuum of \emph{all} reals.
It is also akin to the bifurcation between computable structure theory
and \emph{uniform} computable structure theory.
\end{quote}
\end{minipage}

\section{The main definition}

In the following sections, we will work up to the main definition.
Remember that we want to simultaneously generalize online algorithms
on finite and infinite structures, and in a  general setting
where  the domains might be any kind of structure. Because of this
we will need to tour through \emph{representations} (computational ways of
naming infinite objects), and carefully argue why choices, such as
that of primitive recursion, are made.

\subsection{Representation spaces}

It could be argued that for relational structures we could consider
(isomorphism types of) any
structure $A$ with universe ${\mathbb N}$, and we could consider
$A_n$ to be the induced substructure of $A$ with universe $\{1,\dots,n\}$.
Naturally, we need to assume that this has meaning: and such substructures
exist in all finite cardinalities. Also, if we choose to add function symbols
we would need to only allow a small extension of the structure
based on $\{1,\dots,n\}$. For simplicity, we will stick to relational
structures and use the following terminology.

A class $\mathcal{C}$ of relational structures is called \emph{inductive}
if $A\in {\mathcal C}$ implies $A$ has a \emph{filtration}
$A=\cup_s A_s$
where each $A_n$
is finite, has universe $\{1,\dots,n\}$,
and for all $n'>n$
the substructure induced by $\{1,\dots,n\}$ in $A_{n'}$ is $A_n$.
More generally, for a fixed (Turing) computable function $g$, we say that
$\mathcal{C}$ is $g$-\emph{inductive} if it has a $g$-filtration
meaning that each $A_n$ has universe $\{1,\dots,g(n)\}$.
Here we will sometimes write $O(h(n))$-inductive for the case where $g$ is
$O(h)$. Our language will typically be finite and relational, and $g$ will typically be primitive recursive\footnote{Richard Shore
observed that the punctual case focusses attention upon functions and functional languages, whereas the operator approach seems to tie itself to relational ones. We need some care if the language has function symbols, see \cite{KMM}.}.

We refer to the substructure of $A$ based on $\{1,\dots,n\}$ the substructure
of \emph{height} $h(n)=n$.
In the example discussed above, the height $n$ structures are
the graphs with $n$ vertices. Another example is considered by
 Khoussainov~\cite{KhRandom}
with a height function in his work on random infinite structures.
Natural online structures tend to have  natural height functions.

By abusing notation, we will let
 ${\mathcal C}^{<\omega}$
denote the class of finite substructures of ${\mathcal C}$.
There is also the natural induced topology.
For example, in the graph case this would be compact and have the
totally disconnected topology with basic open sets being the
extensions of graphs of height $n$.

\subsection{Representations}

A \emph{representation} (a naming system)
 of an inductive  class ${\mathcal C}$ of structures
is a (Turing) computable surjective
function 
$\delta:\omega^{<\omega}\to {\mathcal C}^{<\omega}$,
 which acts faithfully in the sense
that $\delta(\sigma)=C_n$ for $|\sigma|=n$ and $h(C_n)=n$,
and if $\sigma\preceq \tau$ then $\delta(\sigma)$ is an induced substructure of
$\delta(\tau).$ Most examples of representations in the literature are witnessed by a primitive recursive
$\delta$. We thus will assume that $\delta$ is primitive recursive throughout.
We can also extend this in the natural way to $g$-filtrations.
Such a $\delta$ induces a map $\overline{\delta}$
from $\omega^\omega\to {\mathcal C},$
namely $\lim\{\delta(\sigma)\mid \sigma\prec x\}$.
We will call $x\in \omega^\omega$ a \emph{name} or a \emph{representation} for $C\in {\mathcal C}$
if $\overline{\delta}(x)=C.$ Note that it is possible
for a   structure $C$ to have a number of different names.

For the time being, we will regard $\overline{\delta}$ as being injective. When it is possible, we will
replace $\omega^{<\omega}$ with $2^{<\omega}$. 
We will consider functions $f:{\mathcal C}_1\to {\mathcal C}_2$
 represented by   functions $F$ acting on representations
$\delta_i: Q_i \rightarrow {\mathcal C}_i$; we of course  require  that $F$ commutes with  $f$ and  $\delta_i$, $i = 0,1$. 

We emphasise that the function $F$ is acting on \emph{strings} which are
\emph{finite objects}. These represent, e.g., graphs. The
 continuity of the action induces a map $\overline{F}$ which is the completion of the finite maps.

\subsection{Online problems} Although our objects of study are not strings, we implicitly identify them with their representations, in accordance with the previous subsection. In particular, if the representation space is compact then our objects can be identified with strings over a finite alphabet. 

Intuitively, to solve a problem  we need to find  a function $f$ which, on input $i$, chooses an admissible solution from the finite set $s(i)$ of ``correct'' solutions.

\begin{defi}
\label{mainn}
A  online problem is a triple $(I, S,  s)$, where $I$ is the space of inputs (i.e.~the filtration) viewed as finite strings in a finite or infinite computable alphabet, $S$ is the space of outputs viewed as finite strings in (perhaps, some other) alphabet,  and $s: I \rightrightarrows S^{<\omega}$ is a (multi-)function which maps  each $\sigma \in I$ to the set $s(\sigma)$ of admissible solutions of $\sigma$ in~$S$. 
\end{defi}
Note that the multi-valued function $s$ does not have to be computable in general.
For instance, for a colouring problem $I$ will be codes for finite graphs and $S$ for finite coloured graphs. Then $s(\sigma)$ will correspond to the collection of all admissible colourings, e.g., such that adjacent vertices are  distinctly coloured.  These colourings will form the space of admissible solutions.

Most natural problems from finite structures will obey the following
convention, which we will consider in this section.
Only in Section \ref{anall} we will consider
more general cases.

\begin{conv}
\label{conv:br}
Unless explicitly mentioned,
 $I$ and $S$ are  compact with a primitive recursive modulus of compactness; i.e., it is primitively recursively branching when viewed as a tree of strings.
Thus, there is a natural primitive recursive
way to transform $I$ into $2^\omega$ (typically not height preserving).
\end{conv}






\subsection{Taking the completion of an online problem} \label{sec:compl}

 A solution $f$ to an online problem $(I,S, s)$ induces  a solution for the
(topological) completion of the initial problem $(I, S, s)$, in the sense that $f$ can be uniquely extended to a functional
$\bar{f} : [I] \rightarrow [S] $. Here $[I]$ consists of infinite strings $\xi$ such that for every $i$, $\xi \upharpoonright i \in I$, and similarly for $[S]$.

In general, in Definition~\ref{mainn} we may also require $\bar{f}$ to satisfy some global property which cannot be always captured by $s$ from Definition~\ref{mainn}. For example, in Section~\ref{sec:punctualstuff}
a solution must be an isomorphism between two presentations of the same infinite graph. In  general, even if at every stage $f(\sigma)$
may be extendable to some isomorphism, the map associated with $\bar{f}$ may fail to be surjective in the limit. Also, in another example in Section~\ref{sec:punctualstuff} we will require our solution to work only if the input is a presentation of some fixed infinite graph, which is also a property of $\bar{f}$ rather than of any finite approximation to it.
In particular, in this case admissibility of $\bar{f}$ cannot be captured by $s$ in Definition~\ref{mainn}; at least not in general.

\begin{conv}\label{conv:gl}
We will refer to such properties of $\bar{f}$ as \emph{global} and will not incorporate them into Definition~\ref{mainn}.
 \end{conv}

\subsection{The definition} Recall Conventions~\ref{conv:br} and \ref{conv:gl}.

\begin{defi} \label{def:maininf0}
A  punctual solution to (a representation of) an online problem $(I, S,  s)$  is a computable function $f: I \rightarrow S$ with the properties:

\begin{description}


\item[(O1)] $f(\sigma) \in s(\sigma)$ for every $\sigma \in I$;

\item[(O2)] If $\sigma\prec \tau$ then
$f(\sigma)\preceq f(\tau)$;

\item[(O3)] $f$ is primitive recursive.
\end{description}
\end{defi}

Condition (O1) says that the output of  $f$ is an admissible solution.  In (O2) we ask for is that each increment of the input yields an
increment in the output,
in the sense that $f(\sigma)$ must be a solution to $\sigma$.

The reader should note that  (O3) is somewhat ambiguous as stated because it may be interpreted in at least two different ways, namely $f$ could be primitive recursive either as a function or as a functional:

\begin{description}
\item[(O3)$'$] the computation of $f(\sigma)$ is based solely on $\sigma$, or

\item[(O3)$''$] the computation of $f(\sigma)$ may ask for an extension $\tau $ of $\sigma$ before it halts.
\end{description}

Indeed, for online computations it would be natural to demand that
we have a primitive recursive timestamp function $g$ and
to compute $f(\sigma)$ we would look at $\sigma' $ of length
$g(|\sigma|)$ extending $\sigma$. In practical computations
lookahead will typically be $g(|\sigma|)=|\sigma|+k$ for some
constant $k$. On the other hand, for a recursion theorist it would be more natural to consider \emph{Turing functionals}
acting on the representation spaces and demand that they are primitive recursive.
By that we mean adding the characteristic function for the infinite string in
the completion of the problem (Section~\ref{sec:compl}) to the primitive recursive scheme of $f$; to be clarified in Subsection \ref{subsec:prf}.
These two general definitions of lookahead (via timestamp and via oracle) are not equivalent when, say,  $I \cong \omega^{<\omega}$.
Thus, we have three natural versions of $(O3)$ which are furthermore provably not equivalent in general.

Luckily, in the next subsection we will prove that, under Convention~\ref{conv:br}, these three versions of the main definition  are equivalent up to a primitive recursive change of notation, and therefore Definition~\ref{def:maininf0} is robust.

\subsection{The robustness lemma}\label{subsec:prf}  As we mentioned above, there are two natural ways of interpreting what it means for $f$ in $(O3)$
to be primitive recursive with a lookahead. We give more details.

In the first definition, we require that it is a Turing functional  that possesses a primitive recursive time-function $t$ which, on every input $\sigma$
outputs the number of steps which $f$ takes to compute $f(\sigma)$. In particular, $t(\sigma)$ bounds  the use of the operator, that is,  the length of $
\tau$ extending $\sigma$ which may be used in the computation of $f(\sigma)$. The length of the output $f(
\sigma)$ will also be bounded by $t(\sigma)$.


The seemingly more general definition of a primitive recursive functional says that,
for each infinite path $x$ through the space of inputs, $f$
is primitive recursive relative to $x=\lim_s \{\sigma\mid
\sigma\prec x\}$. The latter can be formally defined by adding the characteristic function for $x$  to the primitive recursive
schema, and hence would potentially entail that $f(\sigma)$ could be arbitrarily
long for various extensions of $\sigma$.

These two notions are equivalent
in our framework. (Recall Convention~\ref{conv:br}.)

\begin{lem}\label{lem:ptime} For a primitively recursively branching $I$, a Turing functional $f: I \rightarrow S$ possesses
a primitive recursive time-function iff $f$ is a
primitive recursive functional.
Moreover, if $f$ possesses a primitive recursive time-function, then there
is an equivalent online problem $(I',S',s')$
with  a primitive recursive solution
without lookahead.
\end{lem}

\begin{proof} In a different terminology the proof will appear in~\cite{KMM}. A similar formal argument can be found in the appendix of \cite{bsl}.

Suppose the Turing functional $f$ possesses a primitive recursive time function $t$.
Using $t$ as a universal bound on all the searches which may occur in a computation
 with any oracle $x$ extending $\sigma$,
we can transform the general recursive scheme (augmented with the characteristic function $\chi_x$ for $x$)
into a primitive recursive scheme augmented with $\chi_x$. This implication holds in general, i.e., without any
extra assumption on $I$.

Now, assuming $I$ is primitively recursively branching, suppose $f$ is a primitive recursive functional with functional oracle $g$ in
 the most general relativised sense.

 For the base of induction consider the following cases:  $\Phi^g=o$, $\Phi^g=s$,  $\Phi^g=I^n_m$, and $\Phi^g=g$, where $o,s$ and $I^n_m$ are the standard elementary basic functions (e.g., \cite{Rogers}). The first three cases are evident since they do not refer to $g$, while in the case when $\Phi^g=g$ take $t = b$, where $b$ is the primitive recursive branching of~$I$.
 Take $t(x) = \sum_{i \leq x} b(i)$  to make $t$ monotonically increasing in its input.

The inductive step splits into two different cases depending on whether the last iteration is composition or an instance of primitive recursion.

Suppose it is composition,
   $$\Psi^g(\bar x)=\Phi^g(\Theta_1^g(\bar x),\dots,\Theta_m^g(\bar x)),$$
where  $\Psi,\Theta_1,\dots,\Theta_m$  are primitive recursive operators with  primitive recursive time bounds $t_0 ,t_1,\dots, t_m$. As usual, we identify a tuple $\bar x = \langle x_1, \ldots, x_m \rangle$ with its primitive recursive code.
  The time functions $t_0 ,t_1,\dots, t_m$ can be assumed monotonically increasing in their inputs.

 Define a primitive recursive time bound for $\Psi$ by the rule
$$t(\bar{x}) =   t_0 (\langle t_1(\bar x), \ldots, t_m(\bar x) \rangle) + \sum_i t_i(\bar{x}),$$
which can be rewritten into a primitive recursive schema using the standard techniques. Now suppose that $\Psi$ is defined using an instance of primitive recursion, more specifically
   $$
   \Psi^g(\bar x,0)=\Theta^g(\bar x);
   $$
   $$
   \Psi^g(\bar x,y+1)=\Phi^g(\bar x,y,\Psi^g(\bar x,y)),
   $$

\noindent   where  $\Phi$ and $\Theta$ are primitive recursive operators which have corresponding primitive recursive time functions
$t_0$ and $t_1$.  Define $t$ by the rule
 $$
   t(\bar x,0)=t_0(\bar x);
   $$
$$
  t(\bar x,y+1)=t_1(\bar{x}, y, t(\bar{x}, y)) + t(\bar{x}, y);
$$
assuming that $t$ is monotonically increasing in its input this gives the desired upper bound.




\

For the last part of the lemma, recall that
 the tree $I$ is primitive recursively branching. Thus,  we can inductively
form a new
tree $I'$ whose level $n$ nodes are in a (primitive recursive) $1$-$1$ correspondence with the nodes at level  $t(n) = \max\{t(\sigma): |\sigma| =n\}$
in $I$. Then the algorithm $f$ on $I$
works on $I'$ as a strict algorithm. \end{proof}

\









Of course, if $f$ is primitive recursive functional  without lookahead then $f$ can be viewed as simply a primitive recursive function mapping finite strings to finite strings. More formally, we have:

\begin{fact}\label{fact:obT} Suppose $\mathcal{P} = (I, S, s)$ is an online problem. Then the following are equivalent:
\begin{enumerate}
\item $\mathcal{P}$ has a solution witnessed by a primitive recursive function $f$.
\item The completion of $\mathcal{P}$ has a solution witnessed by a primitive recursive operator $\overline{f}$ without lookahead.
\end{enumerate}
\end{fact}

\subsection{Generalisations and refinements of the main definition.}


It is important to understand that primitive recursion smoothens many difficulties related to notation. In particular,
the robustness lemma from the previous section will typically fail for polynomial-time operators.
Thus,  different interpretations of  $(O3)$ in Definition~\ref{def:maininf0} will potentially lead to different refinements of the main definition to more narrow complexity classes. On the other hand, different versions of relativisation (such as general Turing and sub-recursive) will lead to potentially non-equivalent generalisations of the main definition. We will not develop these topics in too much detail, but some notions and notation introduced in this subsection will be important in the later sections.

\subsubsection{Strict solutions, obT operators, and totality } We may want to stick with a given notation $(I,S,s)$
because changing it may be either inconvenient  or computationally too hard.
If the space $I$ is not primitively recursively branching or not even compact, then Lemma~\ref{lem:ptime} no longer holds.
Thus, in this case the
 most general version of Definition~\ref{def:maininf0} becomes ambiguous. In contrast, Fact~\ref{fact:obT} does not rely on compactness of $I$, let alone its primitive recursive branching,  and therefore the stronger strict version of Definition~\ref{def:maininf0} still makes sense even  for non-compact $I$. Thus, the situation described in Fact~\ref{fact:obT} deserves a special attention.

\begin{defi} \label{def:strictdef}

In the simpler situation that no lookahead is allowed in Definition
\ref{def:maininf0}
 we will call $f$ a
\emph{strict} punctual solution.

\end{defi}


By Fact~\ref{fact:obT},
this situation
can be considered an analog of a classical
 ibT-reduction (to be defined),
but acting on compact spaces with primitive recursive branchings with
the branches of level $n$ being the structures of height $n$, instead of $2^{\omega}$.
Classically, ibT refers to an oracle procedure $\Gamma^B=A$ with the use $\gamma(x)=x$ for all $x$,
and here we are identifying sets with their characteristic functions as usual~\cite{DH}.
ibT functionals and the induced reduction have been studied quite intensively~\cite{C,BL,DH,DHL,So} and even used in
(classical) differential geometry \cite{C,NW}.

\begin{defi}
For a fixed filtration representing $I$ we will call such a procedure
$\Gamma$ induced by a strict  online solution an \emph{obT} (online bounded
Turing) reduction.
\end{defi}



 Of course, the classical ibT reduction is usually viewed as working on $2^\omega$. If our space does not have primitive recursive branching
 then we no longer can transform it effectively into a copy of $2^{\omega}$.
But as mentioned earlier,
we see this aspect as a feature of the model, and not a
flaw.
One should expect online-ness to be generally representation dependent, at least to some extent.

\smallskip

Although there are notions of a polynomial-time functional in the literature \cite{Ko},
 Definition~\ref{def:strictdef} is much more convenient if we want to define what it means for a punctual solution to be polynomial time.

We can also use  Definition~\ref{def:strictdef} to give an explicit connection of our definition with
 the above-mentioned approach in \cite{hand,onlinebook} which relies on total (not necessarily computable) functions.
As has been observed in \cite{KMM}, a total function can be viewed as a function primitive recursive relative to some oracle.  More formally, we have:

\begin{fact} For an online problem, the following are equivalent:
\begin{enumerate}
\item The problem has a total strict solution $f$ (in the sense of \cite{hand,onlinebook});

 \item The problem has a strict solution primitive recursive relative to some oracle (in the sense of~\cite{KMM}).
 \end{enumerate}
\end{fact}

In computability theory many arguments tend to be uniform enough to be relativizable to any oracle. The relativisation phenomenon partially explains why the seemingly crude approach via  totality~\cite{hand,onlinebook} often captures some features of online computation.


\









\subsection{General computable and efficient solutions}

\emph{Many} of our results are valid for computable online solutions, and in fact for any total solutions.
For example, certainly a result showing that \emph{no} computable
solution is possible is
very strong.
For example, the proof that online colouring forests from
Gasarch \cite{Gas} requires $\Omega(\log n)$ many colours
shows that no computable $f$ is possible. We arrive at the following generalisation of the main definition.

\begin{defi} \label{def:maininf1}
A  computable solution to (a representation of) an online problem $(I, S,  s)$  is a \emph{Turing computable}  function $f: I \rightarrow S$ with the properties $(O1)$ and $(O2)$ of Definition~\ref{def:maininf0}.




\end{defi}

One potential extra feature which is captured by this generalisation is that it also covers~\emph{partial} solutions
and, potentially, \emph{partial} representations.
We will not make it overly formal and leave this to the reader. For example,
in the case that the representations
are partial, $f$ would perhaps only need to work well on a valid
input (cf.~Convention~\ref{conv:gl}). For example, in the case when $\omega^{\omega}
$ is representing Cauchy sequences, imagine we are seeing an online way to compute some continuous function.
Then we would only need to
produce a solution for those sequences which actually corresponded to
convergent sequences. We could then require this solution be in some sense punctual when restricted to valid inputs.

We could on the other hand make the definition more efficient, for example:

\begin{defi} \label{def:maininf2}
A  polynomial-time solution to (a representation of) an online problem $(I, S,  s)$  is a  punctual strict solution which is furthermore polynomial-time.
\end{defi}
The definition above is of course heavily notation-dependent.
For the look-ahead case the situation becomes even more complex. Although there are definitions of a polynomial-time functional in the literature~\cite{KC,Ko} they tend to be unconvincingly technical. Also, recall that the robustness lemma fails for polynomial-time simply because changing notation tends to be exponential time.
Therefore the definition will also depend on which version  of the main punctual definition we choose to make polynomial-time; recall there were three such versions.

\subsection{Multiple solutions}

Notice that in actual practice, we might also need a further generalisation  of the above. Sometimes we might compute a (bounded) collection of solutions at least one of which is correct at any stage and at height $n$. This occurs in, for example, using automata to compute minimization problems for graphs of bounded pathwidth (or $k$-interval graphs, see section \ref{sec:four}) given the path decomposition. We will be computing a table of $f(k)$ many solutions at each level $n$. For example, for finding maximal clique you would  have a collection of $2^k$ many possible solutions.
However, it appears that a suitable choice of the space of outputs $S$ can cover this seemingly more general case too.

\section{Oracle computation and uniformity}\label{sec:punctualstuff}

\subsection{Graph oracles do not help}





 The main goal of this subsection is to show that
a graph-oracle cannot significantly help in computing a function online.
For that, we consider  online functionals and online oracle computations.

\begin{defi}
We say that $f:2^\omega\to 2^\omega$ is \emph{online computable}
if $f$ has a representation $F:2^{<\omega}\to 2^{<\omega}$, which is online
computable in the sense above, so that
 for all $\alpha\in 2^\omega$, $F(\alpha \uh u(n))=F(\alpha)\uh n$,
where $F(\alpha)=\lim \{F(\sigma)\mid \sigma\prec \alpha\}$
and $u$ is primitive recursive.

\end{defi}

The space $2^\omega$ can be replaced with a primitively recursively branching totally disconnected space.
Identifying $f$ with its representation $F$, we can unambiguously write this as
$f(\alpha\uh u(n))=f(\alpha)\uh n,$
and (in view of Lemma~\ref{lem:ptime}) this should cause no problems in the case of primitively recursively branching spaces of strings.
We may also allow more than one input in $f$.

\begin{nota}
It is natural to write $f^{ \alpha \upharpoonright u(i)}(i)$  instead of $f(\alpha \upharpoonright u(i))$ and view $\alpha$ as an oracle.
The output of $f^{ \alpha \upharpoonright u(i)}(i)$ can also be interpreted as a natural number, when necessary.
\end{nota}

\begin{rem}
 There are obvious refinements of this. For example, it is natural to restrict ourselves to functionals $f$ whose running time is a polynomial in the length of $\alpha$. Also,  having in mind some particularly nice  primitive recursive function $u$,
$f$ is $u$-\emph{online computable} if
$f(\alpha \uh u(n))=f(\alpha)\uh n$. An obvious case is when
$u(n)=n+k$, which would be \emph{online with delay $k$}.
An illustration of this notion can be seen from
Section \ref{real}, where we look at online real valued functions.
We note that addition of reals is online
computable with delay 2, meaning that to compute
the sum of $x$ and $y$ to within $2^{-n}$ needs
$x$ and $y$ with precision $2^{-(n+2)}.$
Similar delay considerations come from other procedures
in polynomial time analysis such as integration
(see \cite{Ko}).
When $u(n) =n$ then the notions can be restated in terms of strict (ibT primitive recursive) functionals, while online with delay $k$ corresponds to  Lipschitz reducibility. Computable Lipschitz reducibility comes from
algorithmic randomness (\cite{DH}, Chapter 9) where it is
shown that if $f$ is online computable Lipschitz acting on $2^\omega$,
then it preserves the Kolmogorov complexity of all sequences in the sense that
for all $n$, $K(\alpha\uh n)\ge ^+ K(f(\alpha)\uh n)$; that is
$K(\alpha\uh n)\ge K(f(\alpha)\uh n)\pm O(1)$.
\end{rem}

We will consider online functionals acting on algebraic or combinatorial structures, e.g., $\alpha$ could be viewed as a description of a finite segment of an infinite structure of some fixed finite relational signature, e.g., a graph. 
 The extensions of $
\alpha \upharpoonright n$ are the finitely many possible relational  structures on $n+1$ elements extending the structure described by $\alpha \upharpoonright n$. The intuition is that $f^{\alpha\upharpoonright u(i)}(i)$ is expected to compute correctly only if $\alpha$ is an initial segment of  a graph $G$. This is a global property; see Convention~\ref{conv:gl}.

\begin{defi}
A function $h : \mathbb{N} \rightarrow \mathbb{N}$ is online computable from the isomorphism type of a structure $G$ if there is an online $f$ such that, whenever $\alpha$ is a description of $G$, $h(i) = f^{\alpha}(i)$.
\end{defi}

In other words, $h$ is allowed to use any presentation of some fixed $G$ as its (online) oracle.

\begin{exa}
To see how much extra computational power algebraic oracles can give, consider the following example. Let $X$ be an arbitrary subset of $\mathbb{N}$, and define $A(X)$ to be an algebraic structure
in the language of one unary function $s$, one unary predicate $p$, and one constant $o$, and which has the following isomorphism type. When restricted to $s$ and $o$, it is  just $\mathbb{N}$ with $s(x) = x+1$ and $o$ interpreted as $0$. Now define $p(x) \iff x \in X$.
Given any presentation $\alpha$ of $A(X)$, we can decide $X$. So, in particular, computation from an isomorphism type is potentially as powerful as just the usual oracle computation.
\end{exa}

In view of the example above, the reader will likely find the theorem below unexpected.   Its proof is however not difficult; it can be viewed as a variation of an argument in Kalimullin, Melnikov, and Montalb{\'a}n~\cite{KMM}.

\begin{thm}\label{thm:or}
A function $h$ is online computable from the isomorphism type of an infinite graph $G$ if, and only if, $h$ is primitive recursive.

\end{thm}

\begin{rem}
It will be clear from the proof below that the result has a natural polynomial time version. The exact definition of a polynomial time  functional is a bit
lengthy; see~\cite{FK,KC}\footnote{The point is that care is needed with
which representations are allowed. Polynomial time functionals for
$(0,1)$ typically use the so-called signed digit representation,
but even for ${\mathbb R}$ there is some problem with the notion  of
the size of the input as discussed in, for instance, \cite{KC}. However,  for
any reasonable representation of graphs of size $n$  this becomes relatively straightforward using, e.g, the standard matrix representation as in \cite{GJ}.}.
We leave the polynomial time case to the reader.
\end{rem}

\begin{proof} By Ramsey's theorem, $G$ either has an infinite clique or an infinite anti-clique; without loss of generality, suppose it is a clique.
Since  $g(i) = f^{\alpha\upharpoonright u(i)}(i)$, where $\alpha$ is \emph{any} representation of $G$, we can assume that the first $u(i)$ bits of $\alpha$ describe a clique. Since the space of all presentations of $G$ is primitively recursively branching, the use $u$ is primitive recursive (see Lemma~\ref{lem:ptime}).
Thus, the oracle can be completely suppressed and the trivial description of an infinite clique can be incorporated into a new procedure $f_0$ which does not use any oracle. On input $i$ the procedure produces a string of length $u(i)$ which describes a finite clique, and then refers to this finite string
(viewed as a partial function) whenever it needs to use the characteristic function of the oracle.
This procedure is easily seen to be primitive recursive (as a function). \end{proof}

Informally, the result says that, from the perspective of online computation, graphs cannot code any non-trivial information into their isomorphism type; i.e., up to a change of their presentation. Both the theorem above and the main result in ~\cite{Dea}
imply that graphs are \emph{not universal for punctual computability} -- a notion which we will not formally define here (see \cite{bsl}).
See Kalimullin, Melnikov, and Montalb{\'a}n~\cite{KMM} for a  generalisation of Theorem~\ref{thm:or} to structures in an arbitrary finite relational language.

\subsection{Interactions with punctual structure theory}

In \cite{bsl} we described the foundations of online structure theory. The main objects in this theory are infinite algebraic structures in which operations and relations are
primitive recursive. As we argued in \cite{bsl}, there are natural strong connections of this new theory and the theory of polynomial-time algebraic structures (see also Alaev~\cite{ALA} and Alaev and Selivanov~\cite{ASel}) with applications to automatic structures~\cite{BHKM}.
Earlier we argued that this kind of punctual structure theory
is akin to Turing-Markov computable analysis, in the objects are
given effectively.
In this paper structures themselves do not have to be primitive recursive. However, the frameworks are closely related via, e.g., Theorem~\ref{thm:pcc} below.

A \emph{presentation} of a countably infinite algebraic structure in a finite language is an isomorphic copy of the structure upon the domain $\mathbb{N}$.  For simplicity, we may assume that the structures in this section are all relational. In this case it becomes consistent with our framework; in particular,  the space of all presentations $I$ of a fixed  structure in a finite relational language is primitively recursively branching.

Each such presentation $\alpha \in [I]$ can be viewed as an isomorphic copy of the structure upon the domain of $\mathbb{N}$. Some of these presentations will be \emph{computable} in the sense that the relations on $\alpha$ will be computable predicates over $\mathbb{N}$.
 It is well-known that a structure may have non-computably isomorphic computable presentations. When we restrict ourselves to primitive recursive presentations and primitive recursive isomorphisms the situation becomes even more complex because the inverse of a primitive recursive function does not have to be primitive recursive. See~\cite{bsl} for a detailed exposition of the theory of punctually categorical structures.

The following notion is not restricted to primitive recursive presentations. A more general  version of the definition below was first discussed briefly in \cite{KMN2} and then also mentioned in \cite{KMN1}. An even more general model-theoretic version of the definition
can be found in~\cite{KMM}.

\begin{defi}\label{def:onlinecc}
A structure $G$ is strongly online categorical if there is an online strict operator $f$ which, on input $\alpha$ and $\beta$ arbitrary representations of $G$ outputs an isomorphism from $\alpha$ onto $\beta$.
\end{defi}
In other words, there exists a primitive recursive functional  $f^{\alpha; \beta}$ with both uses being the identity function, such that
the associated function  $h(i) = f^{\alpha\upharpoonright i; \beta \upharpoonright  i}$ (whose output is interpreted as a natural number)  induces an isomorphism from $\alpha$ onto~$\beta$; recall the latter two are isomorphic copies of $G$ upon the domain $\mathbb{N}$. Equivalently, we could
replace the functional by a primitive recursive function of three inputs $\sigma, \tau, i$ where $|\sigma| = |\tau| =i$ and finite strings are identified with their indices (under some fixed natural enumeration).

The theorem below can be viewed as a variation of another result of Kalimullin, Melnikov, and Montalb{\'a}n~\cite{KMM} on punctual categoricity, but in our strongly online case the proof will be   significantly simpler. Recall that a structure $G$ is homogeneous if for any tuple $\bar{x}$ in $G$  and any pair of elements $y, z \in G$, we have that  $y$ is automorphic to $z$ over $\bar{x}$.

\begin{thm}\label{thm:pcc}
A structure in a finite relational language  is strongly online categorical if, and only if, it is homogeneous.

\end{thm}

\begin{proof} Each homogeneous structure is trivially  strongly online categorical. Now suppose $G$ is strongly online categorical.
Suppose the structure is not  homogeneous, and let $\bar{x}$ be shortest (of length $n$) such that for some $z,y$ we have that $z$ is not in the same automorphism orbit as $y$ over $\bar{x}$.
Construct $\alpha$ and $\beta$ as follows. First, copy $\bar{x}$ into both and calculate the online isomorphism $f$ from    $\alpha \upharpoonright n$ to $\beta \upharpoonright n$. If we identify $\alpha \upharpoonright n$ and $\beta \upharpoonright n$ with $\bar{x}$, then $f$ induces a permutation of $\beta \upharpoonright n$; by the choice of $n$ any permutation of $\bar{x}$ can be extended to an automorphism of the whole structure. Adjoin $z$ to $\alpha$ and find a $y'$ which plays the role of $y$ over $\beta \upharpoonright n$ under any automorphism extending the permutation  $\beta \upharpoonright n \leftrightarrow f(\alpha \upharpoonright n)$. Then necessarily $f(z) = f(y')$, because $f$ has already shown its computation on the first $n$ bits. However, by the choice of $z$ and $y'$, $f$ cannot be extended to an isomorphism no matter how we extend the presentations further.
\end{proof}

Note that we used only totality of the strict functional in the proof.
In the case when the language has functional symbols the theorem no longer holds. Of course, the notion of strongly online and of a presentation will have to be adjusted. But regardless, strong homogeneity will no longer capture the property (whatever it may be exactly).
\begin{exa}
Consider the structure in the language of only one unary functional symbol $s$, and which  consists entirely of disjoint 2-cycles.
Here a 2-cycle is of course a component of the form $\{x, s(x)\}$ where $s(s(x)) = x$ and $x \neq s(x)$.
 According to any reasonable definition of  (strong) online categoricity for functional structures,  this structure has to be (strongly) online categorical. However, it is not homogeneous.
\end{exa}

We leave open:

\begin{prob}
Is it possible to find a reasonable algebraic description of (strongly) online categorical algebraic structures in an arbitrary finite language?
\end{prob}

We suspect that such a description exists, and that the solution will likely boil down to setting the definitions right. If we replace strict with primitive recursive operators in Definition~\ref{def:onlinecc} we will obtain the more general notion of (uniform) online categoricity. With quite a bit of effort
Theorem~\ref{thm:pcc} can be extended~\cite{KMM} to this more general notion, and even beyond.

\section{Weihrauch reduction and online algorithms}

Weihrauch reduction is one of the central notions in computable analysis. It was  named  by
Brattka and Gherardi \cite{BG}. Weihrauch reducibility
 can be viewed as a  natural generalisation
of computable Wadge reducibility~\cite{Wad}.
Henceforth will use $f\le_W g$ to denote Weihrauch reducibility.
$f\le_W g$  has the following intuition.
We have some problem we wish to solve by computing an instance
$f(x)$ of some function $f$.
To do this we produce another instance $x'$
and solve $g(x')$ for $g$, and then convert $g(x')$
back to  of $f(x)$.
In more detail, for functions $f$ and $g$
on $\omega^\omega$-represented spaces $X$ and $Y$,
 $f\le_W g$,
is defined to mean that there are computable $A$ and $B$
 on $\omega^\omega$,
such that for \emph{any} $p_x$, and \emph{any} representation $G$ of $g$,
$$A(p_x,G(B(p_x)))$$
realizes $f$ (i.e. is a name for $f(x)$).
(This is defined here for single-valued functions, but does have a multi-valued
version we won't need.)
This should be thought of as follows for the archetypal case of a
computable metric space.
For a computable metric space,
we take a Cauchy sequence converging to $x$, use $B$ to convert this into a
one converging to $B(x)$, and hence one converging to $g(B(x))$, and finally
using the one converging to $x$ and this one, to one converging to
$A(x,g(B(x))).$

The definition has a number of natural variations; some of these will be discussed below.

\subsection{Weihrauch reduction and incremental computation }\label{miltt}
 In this subsection we establish a formal connection between computable analysis and computer science. More specifically,  we show that a version of Weihrauch reduction borrowed from computable analysis \cite{Wei00} is equivalent to incremental reduction between online problems suggested in Miltersen et al.~\cite{Milt}.

We first state Weihrauch reductions in the online setting.
Suppose $\mathcal{P}$, $\mathcal{Q}$ are online problems.

\begin{defi}
  We say that $\mathcal{P}$ is strongly  Weihrauch reducible to $\mathcal{Q}$, written $\mathcal{P} \leq _{sW} \mathcal{Q}$, if there exist
  Turing functionals $\Phi$ and $\Psi$ such that, whenever $\sigma \in I_{\mathcal{P}}$ is an instance of $\mathcal{P}$, $\Phi^{\sigma} = \tau   \in I_{\mathcal{Q}}$
  is an instance of $\mathcal{Q}$, and whenever $\rho \in s(\Phi^{\sigma}) $ is a solution to $\Phi^{\sigma}$ then $\theta = \Phi^\rho \in s(\sigma)$ is a solution to $\sigma$.

\end{defi}

Here the reduction is strong in the sense that there is a provably more general definition of (plain) Weihrauch reduction  which will be given in due course.
Note that, according to the definition above, all functionals involved are strict, but  this condition can be relaxed  giving a less tight reduction. 

\begin{nota} We write $\mathcal{P}\leq^C_{sW} \mathcal{Q}$ if both strict functionals (in our sense) $\Phi$ and $\Psi$  in the definitions above belong to a complexity class $C$ having sufficiently strong closure properties (e.g., polynomial-time, polylogspace, primitive recursive, etc.).  \end{nota}

\begin{rem}{
The reader might wonder why we will restrict ourselves to strict-type reductions,
or slight variations, for the online setting. The reason is the following.
Suppose that we have two (represented) online problems $I_1$ and $I_2$.
In an online way we want to use $I_2$ to solve $I_1.$
Now suppose that we have some online algorithm for $I_2$. We could take a
$\sigma$ of length $n$ representing an instance $G_n$ of height $n$ of $I_1$,
and convert it into an instance $\sigma'$ of $I_2$, and use
it to produce a solution $s(\sigma')$ of $I_2$, which could be converted back
into a solution $s(\sigma)=A(s(\sigma'))$ of $I_1$. The key issue we will
investigate is how tight the relationships of sizes of the representations
are. Ideally $|\sigma'|=|\sigma|$.
}\end{rem}

A problem $P =(I, O, s)$ is a \emph{decision problem} if $O = \{0,1\}$ and $s$ is merely  a predicate on $I$.  This is the same as to say that
any solution simply decides whether a predicate holds on a string or not.
We say that $\sigma \in I$ is a \emph{positive instance} of $I$ if $s(\sigma) = 1$.
Milterson et al.~\cite{Milt} analysed complexity classes for online algorithms,
and in a slightly more general situation than our \emph{monotone} one where,
for example, the objects only get bigger.
Miltersen et al.~\cite{Milt} investigate online algorithms in which input data may change with time. For example, in a graph a vertex or an edge can disappear. Their reduction takes into account the potential changes of the input.








\begin{defi}
Let $C$ be a complexity class.
A decision problem $\mathcal{P}$ is $C$-incrementally reducible to another decision problem $\mathcal{R}$, denoted $\mathcal{P} \leq^{C}_{incr} \mathcal{R}$, if the following two conditions hold:

  \begin{enumerate}
    \item There is a  transformation $T: I_{\mathcal{P}} \rightarrow I_\mathcal{R}$ in $C$ which maps instances of $\mathcal{P}$ to instances of $\mathcal{R}$ such that
$s_{\mathcal{P}}( \sigma) = s_{\mathcal{R}}( T(\sigma))$ (i.e, $\sigma$ is a positive instance iff its image is a positive instance).
\item There is a transformation $Q$  in $C$ which,  given $\sigma \in I_{\mathcal{P}}$ and the incremental change $\delta$ to  $\sigma$, where $\delta$ changes $\sigma$ to $\sigma'$ of the same length\footnote{That is, $\delta $ is the difference between $\sigma$ and $\sigma'$.},  constructs the incremental change $\delta'$ to $T (\sigma)$ (where $\delta'$ changes $T (\sigma)$ to $T (\sigma')$).
  \end{enumerate}





\end{defi}

\begin{rem} 
We will here only consider $C$ to be the class of polynomial time computable
functions, and
hence use $\le^P_{incr}$ accordingly. Milterson et al.~\cite{Milt}
also considered e.g.~$C$ to be {\sc Logspace}.
 In \cite{Milt} the authors specify the exact time bounds for all computations involved. This is the reason why they need the seemingly redundant part 2 of the definition above. Also,  they look at auxiliary data structure generated for each instance and at the changes induced to the structure. However, from the perspective of general (e.g.)~polynomial time computation this extra information is not necessary since these auxiliary bounds are evidently polynomial time.

\end{rem}

The proposition below shows that $P\leq^P_{incr} Q$ is a variation of Weihrauch reduction from computable analysis which was independently rediscovered by computer scientists.
We note that complexity restricted versions of Weihrauch reducibility were first introduced and studied by Kawamura and Cook in \cite{KC}. Recall that strong Weihrauch reduction is witnessed by a pair of functionals $\Phi$ and $\Psi$.

\begin{fact}\label{thm:incr}
Suppose $\mathcal{P}$ and $\mathcal{Q}$ are online decision problems. Then  $\mathcal{P}\leq^P_{incr} \mathcal{Q}$ iff $\mathcal{P}\leq^P_{sW} \mathcal{Q}$ with $\Psi = Id_{\{0,1\}}$.
\end{fact}

\begin{proof}
Suppose $P \leq^P_{incr} Q$. Then the transformation $T$ from the definition of incremental reduction can be used as $\Phi$ in the definition of $\leq^P_{sW}$.
Since $\sigma$ is a positive instance iff $T(\sigma)$ is, $\Psi = Id_{\{0,1\}}$.

Conversely, suppose $P\leq^P_{sW} Q$ via $(\Psi, Id_{\{0,1 \}})$, where $\Psi$ is a polynomial functional from the space of inputs $I_\mathcal{P}$ of $\mathcal{P}$ to the space of inputs $I_\mathcal{Q}$ of $\mathcal{Q}$.
Then the first part of the definition of incremental reduction follows from the assumption that $\Psi$ is a functional in $C$. By the continuity of $\Psi$ and the fact that we used $Id$ as the second functional, it suffices to deduce a polynomial time bound on the changes in the inputs of $\Psi(\sigma)$ based on the changes in $\sigma$.
But this bound is just a big-O of the bound given by $\Psi$. \end{proof}

Following \cite{Milt}, we can impose specific bounds on the number of steps required for example, calculating $\delta'$ based on $\delta$.  The expectation is that it should be easier  to make the change than to simply recompute $T(\sigma')$ ``from scratch''. All these specialised bounds can also be expressed in terms of strong Weihrauch reduction; we omit details. As an application of Theorem~\ref{thm:incr} and various results in \cite{Milt}, we can obtain a number of polynomial time and polylogtime Weihrauch reductions in the study of online algorithms.

\subsection{Weihrauch reduction and online graph colouring} \label{sec:four}

Before we discuss the role of Weihrauch reduction in online colouring problem we give a brief overview of the latter.

\subsubsection{Online graph colouring}

Many problems can be re-cast as colouring problems, for example
{\sc Bin Packing}. Indeed, colouring can be thought
of as \emph{avoiding configurations}.
In basic graph colouring, we are simply avoiding
an edge connecting vertices of the same colour, but we could instead
avoid, for example, triangles or any finite set of configurations in some
kind of constraint satisfaction problem. However, as
this is an introductory paper we will stick to basic graph colouring.
There is a large literature on this area
such as Kierstead~\cite{hand}.
Graph colouring is quite a flexible tool, and many algorithmic meta-theorems
such as for monadic second order logic (like Courcelle's Theorem (see \cite{DF,Gr})) can be viewed as colouring with constraints. We believe that this
material has great online potential.

We will mention some of this in this subsection.
As an illustrative example, we will online colour
finite or infinite trees and forests. So the objects of interest are forests
being enumerated one vertex at a time.  Along with the set of vertices the enumeration will also need to include the adjacency relation amongst the vertices already enumerated. In other words, at the $(n+1)$-th step the enumeration will provide us an index for the $(n+1)$-th vertex $v_n$ as well as a finite binary string encoding whether $v_nv_j\in E(G)$ for each $j<n$. To represent the space of all enumerations of a finite graph with $n$ vertices, we can use a finite branching tree of height $n$; and so to represent the space of enumerations of all (finite or infinite) graphs we can use a representation $\delta$ with domain a compact subset $T\subset {\omega}^{\omega}$.

The  result  below is an easy
(restated in our notation)
result from the folklore essentially following from
Bean \cite{bean}.
It works for any (not necessarily primitive recursive)
total online procedure.
\begin{prop} For every online algorithm $A$ there is a $\sigma\in T$ of length $2^{t-1}$ such that the respective graph
 ${\delta(\sigma)}$ cannot be coloured by $A$ in fewer than $t$ colours.
\end{prop}

We will write $\chi_A(G_\sigma)$ for the number of colours used to
colour  $G$ when processed
by the online algorithm $A$. The above is nearly optimal, in that we have the following:

\begin{thm}[Lovasz, Saks and Trotter \cite{LST}] There exists an online
algorithm $A$ such that for every 2-colourable graph $G$, if $G$ has $n$
vertices then $\chi_A(G)\le 1+2\log n$.
\end{thm}

This brings us to the notion of a \emph{performance ratio}.
Most algorithms taught in a standard combinatorics class are \emph{offline}. 
This means that given a finite structure $H$ as input, the offline algorithm is allowed to read the whole of $H$ before performing its calculations and giving the output. This is in contrast to an online algorithm, which must produce the next bit of the output after scanning the next bit of the (encoding of) the input $H$. This clearly puts the online algorithm at a disadvantageous position, for an input $H$ can hide critical information in its global structure which an offline algorithm (but not an online one) can see before beginning to write the output. This motivates the definition below where we compare how much an online algorithm is disadvantaged compared to an offline one.


Consider the situation of  an inductive problem in a class
${\mathcal C}$, and suppose we have an optimisation problem.
Then associated with a $\sigma\in I,\tau\in S$ will be a \emph{cost}
function $c(\sigma,\tau)$, measuring the cost of solution $\tau$ for problem instance $\sigma$, where the cost is a certain metric used to judge how good a solution to a problem is. Now the \emph{performance ratio} of an algorithm $f$
is the ratio of  $c(\sigma,o(\sigma))$ with $c(\sigma,f(\sigma))$ where
$o(\sigma)$ is an optimal solution for $\sigma$; meaning the offline solution.

We illustrate this with colourings. In this case, the problem would be a graph, the solution would be a colouring of the graph, and the cost of a solution would be the number of colours used by the colouring; the fewer colours used by a solution, the better it is. Thus the cost of an optimal solution to a given problem $G$ is the offline chromatic number of $G$.

The \emph{offline} chromatic number  of a graph $G$ will be
denoted by $\chi_{\mbox{off}}(G)$ and for forests, we would
have
$$\chi_{\mbox{off}}(G)=2,$$
as it is well-known that trees and forests can be offline coloured in two colours.
\begin{defi}[Sleator and Tarjan \cite{TS}]
The \emph{performance ratio} of an algorithm $A$ in a represented space 
is defined to be
$$r(\sigma)=\frac{\chi_A(T_\sigma)}{\chi_{\mbox{off}}(T_\sigma)}.$$
\end{defi}

Here we are stating the definition for graph colouring but the definition
applies to any online optimisation problem, as above. In the case of
colouring forests, we see that the approximation ratio is
$O(\log (|\sigma|)).$
In the infinite case, the relevant approximation ratio is the growth
rate of $r(\sigma)$ for all paths in the tree $T$ representing the problem.

For example, a graph is called $d$-{\em inductive} (or $d$-\emph{degenerate})
if the vertices of $G$ can be ordered as $\{v_1,\dots,v_n\}$ so that
for every $i\le n$, $|\{j>i\mid v_iv_j\in E\}|\le d$. For example,
by Euler's formula, all planar graphs are 5-inductive. For those who are familiar with graph
theory, $d$-inductive graphs also include all graphs of treewidth $d$, an extremely important class in algorithmic graph theory (see Downey
and Fellows \cite{DF}, for example).
Again note that $d$-inductive graphs have a compact representation.

\begin{thm}[Irani \cite{Ir1,Ir2}]
Let $\sigma$ represent a $d$-inductive graph of height $n$.
Then \emph{first fit} will use at most $O(d\log n)$ many colours
to colour $G_\sigma$. Moreover, for any online algorithm $A$,
there is a $d$-inductive $G_\sigma$ such that
$\chi_A(G_\sigma)$ is $\Omega(d\log n).$
\end{thm}

Sometimes, this growth rate reaches a limit, as in problems with
\emph{constant} approximation ratios.

The classical example is {\sc Bin Packing}, which can be viewed as a graph colouring problem. We can think of bins as colours,
and the objects as having sizes, and the constraint being that
we cannot have more objects of a specific colour than the bin constraint (bin size).
That is,
{\sc Bin Packing}
 takes as input
sizes $a_i\in {\mathbb N}$
and a parameter $V$
(representing each bin size), and assigned colour
$c(a_i)$
subject to $\sum_{c(a_i)=c}a_i\le V$ for each $c$. Here we seek to minimize
the number of colours (i.e., the number of bins used).

Notice that {\sc Bin Packing} is another example of colouring with
constraints.

\begin{thm}[see \cite{GJ}]
First fit gives a performance ratio of
$2$ for online {\sc Bin Packing}.
\end{thm}

\subsubsection{Online reduction} In this subsection we define a new version of Weihrauch reduction, and we also give a non-trivial
example of such a reduction between two distinct online problems.

\smallskip

Let $X$ and $Y$ be spaces represented by names in $2^\omega$ (for convenience).
Again we think of $f$ and $g$ as being solutions for minimisation problems
corresponding to $X$ and $Y$,
respectively.
Thus, for example, we are thinking of $X$ and $Y$ as
inductive structures with filtrations $\{X_n\mid n\in {\mathbb N}\}$
and $\{Y_n\mid n\in {\mathbb N}\}$
respectively. Then the strings of length $n$ represent the structures of height
$n$, and
$f(\sigma)$ will represent a solution to the problem represented by $\sigma$.
Thus they will have an associated \emph{cost}
which in the case of graph colouring is the number of colours used so far, denoted  as $c(\cdot)$.
We will denote
$f_{\mbox{off}}$ and $g_{\mbox{off}}$
as  offline solutions.
That is $f_{\mbox{off}}(\sigma)$ would be the solution to the minimisation
problem $X_n$ of height $n$ with $\delta(\sigma)=X_n, $ and similarly
$g_{\mbox{off}}$.

We state the below for single valued functions, but again there is an
analogous multi-valued version, where the solution produced for $g$
should be within the correct ratio.
The idea of the following is that on input $\alpha\uh n$,
we want to compute (a representation of) $f(\alpha\uh n)$\footnote{We want to avoid explicit representations, but of course we should have $F$ representing $f$ with $F$ acting on $2^{\omega}$,
and for any $\alpha\in 2^\omega$, $\lim_nF(\alpha\uh n)$ realizes (represents)
$f(\alpha)$.}.
To this we will apply the algorithm $B$ to generate an input to
(a representation of) an input for $g$, and then
use the algorithm $A$ to translate this back to give $f(\alpha \uh n)$.
Again we emphasis that this is all working with representations, and should be read this way.

\begin{defi}  Let $f,g$ be functions on $2^\omega$. Then
 $f$ is called \emph{ratio preserving online reducible to
$g$}, $f\le_O^r g$, if there are (type II) online
computable functions $A$ and $B$ with
and a constant $d$,
such that for all $n$,
$$f(\alpha\uh n)=A(\alpha\uh n,g(B(\alpha\uh n)),$$
and the ratio of $c(f(\alpha\uh n))$ to $c(f_{\mbox{off}}(\alpha \uh n))$ is
at most $d$ times the ratio of
$c(g(B(\alpha\uh n)))$ to $c(g_{\mbox{off}}(B(\alpha\uh n)))$.

\end{defi}

The fact below isolates the most important feature of the reduction.

\begin{fact}\label{fact:red}
If $f\le_O^r g$ then , for some $d>0$,  $\dfrac{c(f\uh n)}{c(f_{\mbox{off}} \uh n)} \leq d\cdot\dfrac{c(g\uh n)}{c(g_{\mbox{off}}\uh n)}$.
\end{fact}


To give a non-trivial example of an online reduction we need several definitions.

\smallskip

In classical colouring, Kierstead investigated online colouring
of {\sc Interval Graphs}. A graph $G=(V,E)$ is called a $k$-interval graph
if each vertex $v$
of $G$ can be represented by a  closed subinterval of $[0,1]$
such that if $I_v$ represents $v$ and $I_w$ represents $w$,
then if $vw\in E$, $I_v\cap I_w\ne \emptyset,$ such that
the largest number of intersecting intervals (the \emph{cutwidth})
is at most $k$.
These are exactly the graphs which have \emph{Pathwidth} $\le k$, a
graph metric coming from the Robertson-Seymour minors project (see \cite{RS86a,DF}).

\begin{defi}
Let $ColInt_k$ denote the online problem of colouring a $k$-interval graph. (We leave the precise representation of the problem to the reader.)
\end{defi}

The other online problem is on covering of an interval partial ordering by chains.
 A partial ordering $(P,\leq)$ is called {\em an interval ordering}
if $P$ is isomorphic to $(I,\leq)$
where $I$ is a set of intervals of the real line
and $x\leq y$ iff  the right endpoint of $x$ is left of the
left endpoint of $y$. Interval orderings can be
characterised by the following theorem.

\begin{thm}[Fishburn \cite{Fi}]
\label{fishburn}

Let $P$ be a poset. Then the following are equivalent.
\begin{description}
\item[(a)] $P$ is an interval ordering.
\item[(b)] $P$ has no subordering isomorphic to
${\bf 2}+{\bf 2}$ which is the
ordering of four elements with
$\{a,b,c,d\}$ with $a<b$, $c<d$ and no other
relationships holding.
\end{description}

\end{thm}

The width of an interval ordering $(P,\leq)$ is defined
naturally to be the minimum over all presentations of the maximum number of intervals covering some point of $[0,1]$.
Given an interval ordering $(P,\leq)$ of width~$k$, our goal is to cover it with as few chains as possible; the chains do not have to be disjoint.

Recall that a \emph{chain} in a partial ordering
is a $\le$-linearly ordered subset. A collection of chains
$\{C_1,\dots,C_q\}$
\emph{covers} $(P,\le)$
if each element of $P$ lies in one of the chains.
An \emph{antichain} is a collection of pairwise $\le$-incomparable elements.

\begin{defi}
Let $ChInt_k$ denote the online problem of covering  an interval ordering $(P,\leq)$ of width~$k$  by  chains (which are not necessarily disjoint). We leave the precise representation of the problem to the reader.
\end{defi}

The theorem below gives a non-trivial example of an online ratio-preserving reduction between online problems. The proof of the theorem below is essentially an analysis of the clever argument given in Kierstead and Trotter \cite{KT}.

\begin{thm}\label{thm:O}  For any positive $k \in \mathbb{N}$ there is an online solution $g$ to $ChInt_k$ with a constant performance  ratio
which can be transformed into an online solution $f$ to $ColInt_k$ with the property $f\leq^r_O g$ via a constant $d=1$.


\end{thm}

\begin{cor}[Kierstead and Trotter \cite{KT}] There is an online algorithm to colour $k$ interval graphs with a constant competitive ratio.  
\end{cor}

\begin{proof}[Proof of Corollary] Kierstead  and Trotter \cite{KT}
 showed that every 
online interval ordering
of width~$k$ can be online covered by
$3k-2$ many chains.

Since $ColInt_k  \leq^r_O ChInt_k$ and is witnessed via a reduction with constant $d =1$,
it remains to apply Fact~\ref{fact:red}. \end{proof}

\begin{proof}[Proof of Theorem~\ref{thm:O}]
The basic idea is quite simple. Take our online $k$ interval graph,
turn it into an online interval ordering of width~$k$, and then
consider that chain covering as a colouring. However, to see that this idea works, we need to argue that
there is an online solution $g$ to the interval chain covering problem
which uses only the
information about comparability of various elements, and not their ordering.

We first prove the following. Suppose that $(P,\leq )$ is a online interval ordering
of width~$k$. Then $P$ can be online covered by
$3k-2$ many chains. We need the following lemma whose proof is
fairly straightforward. For a poset $P$, and subsets $S,T$,
we can define $S\leq T$ iff
for each  $x\in S$ there is some $y\in T$
with $x\leq y$. (Similarly $S|T$ etc.)

\begin{lem}\label{km1}
If $P$ is an interval order and $S,T\subset P$ are maximal
antichains the either $S\leq T$ or $T\leq S$.\end{lem}

The algorithm for chain covering uses induction on $k$.
We consider the vertices as $1,2,\dots$  with $p$ added at step $p$.
If $k=1$ then $P$ is a chain, and there is nothing to prove.
Suppose the result for $k$, and consider $k=1$.
We define $B$ inductively by
$$B=\{p\in P: \mbox{width}(B^p\cup \{p\})\leq k\}.$$
Here $B^p$ denotes the amount of $B$ constructed by step $p$ of the
online algorithm.
Then $B$ is a maximal subordering of $P$ or width $k$.
By the inductive hypothesis the algorithm will have covered
$B$ by $3k-2$ chains.
Let $A=P-B$. Now
it will suffice to show that $A$ can be covered by 3 chains, and then
these will be covered by the greedy algorithm.

To see this it is enough to show that
every elements of $A$ is incomparable with at most two other elements of
$A$.
Then the greedy algorithm  will cover $A$, as we see elements {\em not} in~$B$.

\begin{lem} The width of $A$ is at most 2.\end{lem}

\begin{proof}
To see this, consider 3 elements
$q,r,s\in A$. Then there are antichains $Q,R,S$ in~$P$
of width $k$
with $q|Q$, $r|R$ and $s|S$.
Moreover these can be taken as maximal antichains.
Applying Lemma \ref{km1}, we might as well suppose
$Q\leq R\leq S$.
Suppose that $r|q$ and $r|s$. Then
we prove that $q<s$.
Since $q|r$ and width$(P)\leq k+1$,
there is some $r'\in R$ with $q$ and $r'$ comparable.
Since $q|Q$, $r'\not\in Q$.
Since the width of $B$ is $\leq k$,
there is some $q'\in Q$ $q'$ and $r'$ comparable.
Since $Q\leq R$, there is some $r_0\in R$ with
$q'\leq r_0$.
Since e$R$ is an antichain, $q'\leq r'$.
Since $q|q'$, $q\leq r'$.
Similarly, there exists
$r''\in R$ with $r''\leq s$.
Since $P$ does not have any ordering isomorphic to ${\bf 2}+{\bf 2}$,
we can choose $r'=r''$, and hence $q<s.$\end{proof}

Now we suppose that
$r,q,s,t$ are distinct elements of $A$ with
$q|\{r,s,t\}$.
Then without loss  of generality $r<s<t$ since the width of $A$ is
at most 2.
Since $s\in A$ there is an antichain $S\subset B$ of length $k$\
with $s|S$.
Since $s|q$, and width$(P)\leq k+1$,
$q$ is comparable with some element $s'\in S$.
If $s'<q$, then $s'|r$ and hence the suborder $\{s',q,r,s\}$ is
isomorphic to ${\bf 2}+{\bf 2}$.
Similarly, $q<s'$ implies
$s'|t$ and then the subordering
$\{q,s',s,t\}$ is isomorphic to
${\bf 2}+{\bf 2}$. Thus there cannot be 4 elements
$r,q,s,t$  of $A$ with
$q|\{r,s,t\}$.
Hence $A$ can be covered by 3 chains.

It is easily see that the procedure above uses only comparability of intervals. Thus, the theorem follows. \end{proof}

\begin{prob} Investigate online reduction between online algorithms in the literature.
\end{prob}

We also expect that the online reduction may lead to new online algorithms based on the already existing ones.

Also graphs with constrained decompositions such as those of bounded treewidth,
pathwidth, clique-width, etc have been extensively studied in the literature, and particularly combine well with algorithmic meta-theorems (see e.g.
Downey-Fellows \cite{DF}, Flum and Grohe \cite{FG}, Grohe \cite{Gr} for a sample).

One example is given by $k$-interval graphs met above which are those of
pathwidth $\le k.$ A $G$ of pathwidth $k$  has  a
\emph{path decomposition} which is a collection
of sets of vertices $V_1,\dots,V_n$
all of size $\le k+1$
such that
for all vertices $v\in V(G)$, there is at least one $i$ with $v\in V_i$,
if $xy\in E(G)$, then for some $i$, $\{x,y\}\subseteq V_i$
and finally if $x\in V_i$ and $x\in V_j$ (with $i<j$)
the for all $q\in [i,j]$, $x\in V_q$. The last property is
called the \emph{interpolation property}, and
says that pathwidth is kind of a measure of how far you are from
being either a grid or a clique.

Now \emph{given} such a path decomposition, and some
optimisation property we want to solve (such as for the largest clique),
if the property is definable in monadic second order logic (even with
counting), then
we can solve the problem by dynamic programming (actually using special automata)
beginning at $V_1$ and finishing at $V_n$ by the methods of Courcelle \cite{DF,FG,Gr}.

\begin{prob}
Investigate the extent to which this dynamic programming is online. Presumably,
it will be online for properties defined by monadic second order counting logic with counting modulo some kind of delay.
\end{prob}

Moreover, as we have seen above for the special case of colouring above, we get a constant ratio approximation algorithm, for a graph of pathwidth $k$,
no matter how we are given the online presentation. The difference is
that if we are a given a path decomposition as the presentation, then $k+1$
colours will suffice. But perhaps the methods for colouring are more general.
The point is that graphs of bounded pathwidth have very constrained structure.

\begin{prob} Investigate the approximability of monadic second order definable
properties on graphs of bounded pathwidth, but given as arbitrary
online presentations.
\end{prob}

The same can be asked for graphs of bounded treewidth
which has the same definition as pathwidth, but the structure of the
decomposition is a tree and not a path. These also have
dynamic programming algorithms, but are always
\emph{leaf to root}, whereas even given a tree decomposition as an
online root to leaf structure, presumably some kind of
algorithm will work, but it will no longer be automatic.
This seems a great area to pursue.

Also related seems the idea of online parameterised problems \cite{DM,DM1},
where we want an online solution to a problem with a fixed parameter.
For example, $k$-{\sc Vertex Cover} asks for a collection
of vertices where each edge of a graph includes at least one of the
vertices, and this is polynomial time for a fixed $k$.
This is also online polynomial time for a fixed $k$ by the following simple
algorithm (so long as we are allowed $2^k$ many possible solutions).
We can use the following simple method of building a tree of height $2^k$
by taking an edge, and branching on that edge, and then deleting the covered edges, and repeating. This process is also online. In \cite{DM,DM1}
Downey and McCartin showed that the online view brings to the
other parameters
such as what they call \emph{persistence}
which characterises the extent to which a path decomposition does not resemble
a fuzzy ball. The point is that online algorithms point at new parameters of
a problem which deserve attention, in the same way that parameterised
complexity showed that parameters allow a more fine grained
understanding of the computational complexity of a  combinatorial problem.

\section{$\Delta_2^0$ processes, finite reverse mathematics, and Weihrauch reduction}
\label{delta}
Imagine we are in a situation where the data we are dealing with
is so large that we cannot see it all. At each stage $s$ our goal is to build a solution $f$ to some problem.
But there might be no hope of giving a fixed solution at each stage $n$, and  like a Triage Nurse making an ordering for patients to obtain medical attention,
we would update our solution as more information becomes available.
So for each $n\le s$ we would be computing $f(n,s)$ from the finite information
$\sigma$ with $|\sigma|=n.$
For simplicity we state the next definition for combinatorial problems
with totally disconnected representations, and take $2^\omega$ as
the representing example.

\begin{defi}
A \emph{limiting online algorithm}
on $2^\omega$ is a computable function $A$ such that for each
$s$, $A(\alpha\uh s)$ computes a string $\{f_A(n,s)\mid n\le s\}$
such that
$\lim_s f_A(n,s)$
exists for each~$n$.

 As usual we would have $A(\alpha\uh g(s))$ for the
$g$-online version.

\end{defi}

We can then compare combinatorial problems by how fast their limits
converge.

\begin{defi} We say that algorithm $A\le_{O,lim} B$ if
there is an online Weihrauch reduction of $A$ to $B$ such that
$f_B(n,s)=f_B(n,t)$ for all $t\ge s$ implies
$f_A(n,s)=f_A(n,t)$ for all $t\ge s$.
\end{defi}

This gives a fine grained measure of the complexity of combinatorial problems.
For example, consider the
``theorem'' that every finite binary tree of height $n$ has a path of length
$n$. Then we can consider the existence of a uniform function
$A$ which takes a given binary tree of
height $n$ to a path.
This is an online limit problem where the underlying space $X$ is
that with nodes generated by
the collection of binary trees of height $n$ at level
$n$.
The completion of this will represent paths through infinite binary
trees.

\begin{rem} We could argue that the Reverse Mathematics principle
$WKL_0$ which states that every infinite binary tree has a path,
is equivalent to
the statement that there is a  limiting online algorithm for
finding paths which works on~$X$.
We call this \emph{limiting online paths.}
\end{rem}

A binary tree $T$ of height  $n$ is called \emph{separating} if
for each $j\le n-1$, for any node $\sigma$ on
$T$ of height $j$, and $i\in \{0,1\}$,
if $\sigma * i$ does not have
an extension in $T$
of height $n$, then for all $\tau$ of length $j$, neither
does $\tau * i$. Let $X_S$ be the totally disconnected space representing
the collection of all separating finite trees.
The following is a online interpretation and refinement of the
classical fact that Weak K\"onig's Lemma is
equivalent to Weak K\"onig's Lemma for separating classes.

\begin{prop} There is a $(2^{n+1}-2)$-limiting online reduction
which finds limiting online paths in $X$ from those in $X_S$.
\end{prop}

\begin{proof} We remind the reader of  how this proof works.
Suppose we have a tree $T_s$ of height $s$.
In an online fashion,
we will generate a tree $H$ of height
$2^{s+1}.$ This is done inductively. At step 1,
we can think of the nodes labeled $0$ and $1$ in $T$ as being
represented by $0$ and $1$ in $H$.
At step 2, in $T$ it is possible for us to have
$00,01,10,11$ and these are represented by 4 levels in $H$,
with height 2 representing $00,$ level 3 $01,$ level 4 $10,$ and level 5 $11.$
Now we continue inductively.
This makes level $n$ of $T$ correspond to
trees of height $2+4+\dots+2^n=2^{n+1}-2$.
As the construction proceeds, if some
$\sigma$  fails to have an extension at length $s$, in $T_s$, there will be some
shortest $\sigma'\preceq \sigma$ which fails to have
a length $s$ extension in $T_s$.
Then in $H_s$, we don't extend to length $s$ (from length $s-1$)
all paths corresponding to $\nu * j$ with $j$
representing $\sigma'$ in $H_{s-1}.$

Consider any limiting online algorithm for finding  a path
for path $\alpha$ corresponding to $H$, in $X_0$,
This naturally and in a online way allows us from level $2^{s+1}$ to generate
an online path in $T_s$, and is clearly a limiting online reduction.
\end{proof}

\begin{prob}
Figure out the smallest
$g$ in place of $2^{s+1}$ in the reduction above, which would give a \emph{precise measure}
of how tight the reverse mathematics relationship is.
\end{prob}


There seems a whole research programme available here.
For example, we could be given an online bipartite graph
$B_\sigma$ for $\sigma\prec \alpha$.
We either have to build a complete matching or demonstrate that Hall's
condition fails.
One representation of this problem
will involve a compact space
where the nodes are   bipartite graphs of height $2n$, say,
and  where the paths
all represent graphs which obey Hall's condition.
The online operator will act on this compact tree of representations for graphs
$B_\sigma$.
Now as the process goes along, we might have to update the solution at hand.
That is, the online process has $B_\sigma\mapsto M_\sigma$,

One intriguing example is that of finding a basis in a vector space.
In the case that
the vector space is over the rationals, then presumably this will
correlate to some principle like ACA$_0$. But consider a finite field such as
GF(2). We know that RCA$_0$ proves that we can find  a basis for a vector space
for this field.
But it is not hard to construct an online vector space over GF(2)
for which there is no
online algorithm to do this, unless we have a computable delay.
Comparing the online complexity of such problems with such computable
delay would see to give significant insight into the fine structure of reverse
mathematics. In this particular case, we also note that
a polynomial time algorithm for finding a basis of a polynomial time
vector space was proven to be equivalent to $P=NP$
suggesting intriguing connections with complexity theory.
There is some relevant work by Hirst and Mummert \cite{HirstMum}, who have proved that finding the basis of a vector space has the Weihrauch complexity of lim, i.e., is on the second level of the Borel hierarchy.

We remark that there are many processes that have been investigated
and fall under the model we have introduced. One such example is
algorithmic learning theory, such as $EX$-learning (Gold \cite{Gol}).
Here one is presented with $a_0,a_0,\dots$ values for
a function $f(0), f(1),\dots,$
and we need to eventually print out an index for $\varphi_e=f$ from
some point onwards. This is clearly an example of an online algorithm, and fits into this section as a limiting algorithm.
There are interesting connections between these ideas and reverse mathematics; see, e.g., \cite{Br1,Br2, Br3, HoJSt}.
For online learning in computer science, see~\cite{Sh}.

  Another area which could be incorporated would be asynchronous computing.
Here we have a series of agents $A_1,\dots, A_k$ communicating
through asynchronous channels, and attempting to
compute a set of functions $f_1,\dots, f_k$, where there might
be e.g. some kind of crash failure meaning that one of the agents
dies and stops sending signals. For example, the
{\sc Consensus} problem asks for all the $f_i$'s which have not crashed to
give the same value. A \emph{run} could be represented in
a space of possible communications and failures.
There are a number of reductions which have been
produced in this area, showing that Consensus is a certain
kind of minimal failure, and other problems can be solved if
Consensus can (Chandra and Toueg \cite{CTo}).
It would be interesting to see if these results can be placed in the hierarchy
of online limiting reductions, since they appear
to look like online limiting reductions.

Finally, one exciting possibility would be to include randomization
in this setting. Randomized online algorithms are quite common in
practice (see e.g. Albers \cite{Albers}). For this we could
use the theory of algorithmic randomness (see \cite{DH,LV,Nies}) easily.
For example, an online algorithm with randomized advice (i.e. representing
a coin toss at each stage) could be done via (using $2^\omega$ as  a
representative space) by considering
online algorithms from
$2^\omega\times 2^\omega
\to S$, with $S$ some solution space, with the first copy of
$2^\omega$ representing the problem, the second representing
``advice'' strings and $S$ the solution space.
The online algorithm could take $(\sigma,\tau)\to s_n$, and
would run on extensions of $\tau$ provided that
$[\tau]$ avoids some algorithmic randomness test, such as
a Martin-L\"of test. Using oracles we could also tie this
to the  theory of algorithmic randomness using the
``\emph{fireworks}'' method of Shen (see
Bienvenu and Patey \cite{BP}). Similar approach has been implemented for offline algorithms in \cite{BGH}, and for several interesting results more closely related to our online setting see~\cite{BHK}. In the online situation, these ideas  remain to be further explored.

\section{Real functions.}
\label{real}
\label{anall}

So far all objects of study have been discrete and spaces
compact. However, there is a perfectly reasonable
extension of these ideas to continuous objects such as the space of continuous functions on the unit interval.
There has been a lot of work on complexity theory of real functions; see, e.g., Ko ~\cite{Ko}.
In terms of applications, a natural object of study would be online
analysis; analytic processes which run quickly and only use local knowledge
of the precision of the inputs. As we observe below with
natural representations, addition of reals $x,y$ with precision $2^{-n}$ only
needs $x$ and $y$ to within $2^{-(n+2)}$. Integration and other
standard processes have similar commentry, but we leave this to a later paper.
Also there are other online processes on non-compact spaces, such as EX-learning, or the KC theorem discussed earlier.
We also defer discussion of such topics for later papers, and here stick
to analysis.

The main goal of this section is to demonstrate the role of primitive recursion as a useful abstraction.
The content of this section is not technically hard, but
one can easily imagine a much deeper general framework that could emerge from these basic ideas.

Recall that a Cauchy sequence $(r_i)_{i \in \mathbb{N}}$ of rationals is fast if $|r_i - r_{i+1}| < 2^{-i}$, for every $i$.
These are the \emph{names} which represent the space.
 A function $f: [0,1] \rightarrow \mathbb{R}$ is \emph{computable} if
   there is a Turing functional~$\Phi$ such that, for each $x \in [0,1]$ and for every fast Cauchy sequence $\chi$ converging to $x$, the functional~$\Phi$ enumerates a fast Cauchy sequence for $f(x)$  using $\chi$ as an oracle.
In particular, using the terminology, we would be
generating a representation of the function via names of Cauchy sequences
in such a way that it is representation independent.
That is, $(\Phi^\chi(n))_{ n\in \mathbb{N}}$ is a fast Cauchy sequence for $f(x)$.
This in particular means that, on input $(r_i)_{i \in \mathbb{N}}$, the use of $\Phi^{(r_i)_{i \in \mathbb{N}}}(j)$ corresponds to $\delta$ when $\epsilon = 2^{-j+1}$ in the standard $\epsilon$-$\delta$ definition of a continuous function.

It is well-known that Weierstrass approximation theorem is effectivisable in the sense of Turing computability~\cite{PourElRich}.  This means that a function $f: [0,1] \rightarrow \mathbb{R}$ is computable iff  there is a computable sequence of  polynomials $(p_i)_{i \in \mathbb{N}}$ with rational coefficients with the property $${\rm sup}_{x \in [0,1]} |f(x) - p_i(x) | < 2^{-i},$$
for every $i.$

We have seen that the most general definition  of being online for
combinatorial structures involves being $g$-online for some primitive
recursive function $g$. That is, there is a translation between
using $g(n)$ many bits of $\alpha$ to compute $n$ bits of
$f(\alpha)$. We have also seen that for most natural online situations,
we can translate this to a wider tree where
$\alpha'\uh n$ represents $\alpha\uh g(n)$, so we can use  strict (ibT primitive recursive)
procedures. It is not completely clear if this is natural in the setting
of analysis, since we might wish to stick to \emph{standard} representations of the spaces, like $2^\omega$ and $\omega^\omega$, as above.

We first consider the most
general setting where we allow $g$-online for a primitive
recursive $g$, so using $g(n)$ bits to decide the output for length $n$.
We will call this \emph{punctually} computable.
In this case, there are two natural  definitions of what it would mean for such an $f$ to be ``online'' computable in the most general sense of primitive recursion.
The first notion is the most straightforward sub-recursive version of the standard definition.

\begin{defi}
A function $f: [0,1] \rightarrow \mathbb{R}$ is \emph{punctually computable} if
   there is a primitive recursive functional~$\Phi$ such that, for each $x \in [0,1]$ and for every fast Cauchy sequence $\chi$ converging to $x$, the functional~$\Phi$ enumerates a fast Cauchy sequence for $f(x)$  using $\chi$ as an oracle.
\end{defi}

By restricting ourselves to dyadic rationals, we can assume that fast Cauchy sequences come from a compact totally disconnected space of the names of dyadic rationals in $[0,1]$. Thus, Lemma~\ref{lem:ptime} can be applied to ensure that there is no ambiguity in the notion of a primitive recursive functional in this case. In particular, the definition has a natural polynomial-time version which we omit (see \cite{Ko}); the same applies to any natural complexity class which may be of interest.

The second version filters through the theorem of Weierstrass. It views $f$ as a primitive recursive point in the metric space $(C[0,1], \rm sup)$
rather than as a functional.

\begin{defi}
A function $f: [0,1] \rightarrow \mathbb{R}$ is \emph{uniformly punctually computable} if
   there is a primitive recursive function which on input $i$ outputs (the index of) a polynomial $p_i$ with rational coefficients such that
   ${\rm sup}_{x \in [0,1]} |f(x) - p_i(x) | < 2^{-i}$.
\end{defi}
Clearly, there is a natural polynomial-time modification of the definition above which we omit.

\

Every uniformly punctually computable $f$ is punctually computable.
Are these two definitions equivalent?
It is not completely evident why Weierstrass approximation theorem should hold primitively recursively.
Indeed, in the standard Turing computable proof  we would  wait for a cover of $[0,1]$ by $\delta_i$-balls $B_i$
such that $f(B_i)$ has diameter $< \epsilon$, for every $i$. It seems that even when $f$ is punctual
this search could be unbounded.

Nonetheless, the theorem below shows that these definitions are equivalent. This result  is not really new. With some effort its proof can be extracted from~\cite{Ko}, but the book is mainly focused
 on polynomial time and exponential versions of the definitions above. There is much combinatorics specific to complexity theory which
 significantly obscures the idea behind the proof.
Primitive recursion strips away complex counting combinatorics thus clarifying the idea.

\begin{thm}\label{thm:anal}
Every punctually computable $f: [0,1] \rightarrow \mathbb{R}$ is uniformly punctually computable.
\end{thm}

 \begin{proof}[Proof sketch] The idea here is similar to that in the proof  of Lemma~\ref{lem:ptime}.
Fix $n$ and consider the functional $\Psi^x_n = \Phi^x(n)$ which uniformly primitively recursively outputs the first few bits of $f(x)$ up to error $2^{-n}$, for any input $x$. Since $\Psi_n$ is given a primitive recursive scheme (with parameter $n$), we can work by induction on the complexity of the scheme and emulate all its possible computations at once, as in Lemma~\ref{lem:ptime}. Since the space of dyadic presentations of rationals is
primitively recursively compact, this will lead to a primitively recursively branching tree of possible computations whose height is determined by the syntactical  complexity of the primitive recursive scheme.
By the choice of $\Psi_n$, one of these computations must work for an arbitrary  $x \in [0,1]$.
Thus, we have primitively recursively  calculated an open cover $[0,1]$ by basic open intervals $J_1, \ldots, J_k$, such that whenever $x, y \in J_i$ we have $|f(x)-f(y)| < 2^{-n+1}$. If $z_i$ is the center of $J_i$, then define (the graph of a) piecewise linear function $h_n$ by connecting
points $(z_i, \Psi_n^{z_i} ) $ and $(z_{i+1}, \Phi_n^{z_{i+1}} ) $, $i = 1, \ldots, n$-$1$. Note that the values of the $\Phi_n^{z_{i}}$ have already been calculated. Since the intervals are overlapping, this piecewise linear function $h_n$ approximates $f$ with precision $2^{-n+2}$. We can primitively recursively smoothen $h_n$ by replacing it with a polynomial $p_n$ such that ${\rm sup}_{x \in [0,1]} |p_n (x) - f(x)| < 2^{-n+3}$.  \end{proof}

See Chapter 8 of \cite{Ko} for a detailed analysis of the polynomial-time versions of Weierstrass approximation theorem.
Recall that in the proof sketch above we generated the tree of possible computations. For a polynomial-time operator this tree may be exponentially large at worst. This difficulty cannot be circumvented and the polynomial-time analogy of the theorem above \emph{fails} as explained in great detail in~\cite{Ko}.

We see that punctual analysis fits somewhere in-between computable analysis and polynomial-time analysis, and there is likely much   depth in the subject. Such a theory could provide us with a stronger technical link between computable and feasible analysis. Some basic foundations of elementary primitive recursive analysis was established in the 1950s and the 1960s;
we cite \cite{Sp1,Sp2,Cleave} and the book \cite{Good}. Nonetheless, is seems there has been no recent dedicated study of primitive recursive continuous functions.  Primitive recursive presentations of analytic separable spaces (such as, say, the Urysohn space) have not been systematically studied either.

\begin{prob} Develop primitive recursive (``punctual'') analysis.
\end{prob}

Now in the case that we want to look at the strictly online
model, we are stuck with using, for instance, the bit representation of
a real $x$, and would be working, for example, with
$2^\omega$.
Then to compute $f(x)$ with precision $2^{-n}$ we would need
$x\uh n$.
We might ask for delay $k$ so might use $2^{-(n+k)}.$
Now in this case, we see that, for example
addition is online (on $2^\omega\times 2^\omega$) with delay 2,
and if $f$ is a given online computable function
which is bounded then $\int_0^xf(x)dx$ would also be
online computable with delay 2.
We remark that this model would seem to be one emulating
classical numerical analysis.
We cite  \cite{Wei91} and Chapter 7 of \cite{Wei00} for some closely related results in computable analysis, and see
\cite{Sch, Tri, Mul} for results on online arithmetic in computer science.



\begin{thebibliography}{99}


\bibitem{ALA}
P.~E. Alaev.
\newblock Structures computable in polynomial time. {I}.
\newblock {\em Algebra Logic}, 55(6):421--435, 2017.

\bibitem{Alaev-II}
P.~E. Alaev.
\newblock Structures computable in polynomial time. {II}.
\newblock {\em Algebra Logic}, 56(6):429--442, 2018.



\bibitem{ASel} Alaev, P., and Selivanov, V., Polynomial-Time Presentations of Algebraic Number Fields. Computability in Europe 2018, (2018) 20-29.



\bibitem{Albers} Albers, S., Online Algorithms: A Survey, in \emph{Mathematical Programming},  Vol. 97, (2003) 3-26.

\bibitem{AshKn}C.~Ash and J.~Knight. \newblock {\em Computable structures and the hyperarithmetical hierarchy},  volume 144 of {\em Studies in Logic and the Foundations of Mathematics}.\newblock North-Holland Publishing Co., Amsterdam, 2000.


\bibitem{avi}
J.~Avigad,  \emph{Inverting the Furstenberg correspondence},
American Institute of Mathematical Sciences,  2012.


\bibitem{BL}
Barmpalias, G. and Lewis A., A c.e. real that cannot be sw-computed by any $\Omega$-number, \emph{Notre Dame Journal of Formal Logic,} 47 (2006), 197–209.

\bibitem{bsl} Bazhenov, N.,  Downey, R., Melnikov A., and Kalimullin, I.
Foundations of online structure theory, \emph{Bulletin of Symbolic Logic,}
Vol. 25
(2019), 141-181.

\bibitem{BHKM} Bazhenov, N.,  Harrison-Trainor, M., Kalimullin, I., Melnikov A., and Ng, K.M. Automatic and polynomial-time algebraic structures, \emph{J. Symbolic Logic}. In press.


\bibitem{bean} Bean, D., Effective Coloration,
\emph{J. Symbolic Logic}
    Volume 41, Issue 2 (1976), 469-480.

    \bibitem{BP} Bienvenu, L., and Patey, L.,
    Diagonally non-computable functions and fire-works. \emph{Information and Computation,} 253 (part 1):64–77, 2017.


\bibitem{onlinebook}
Allan Borodin and Ran El-Yaniv.
\newblock {\em Online computation and competitive analysis}.
\newblock Cambridge University Press, New York, 1998.




 \bibitem{Bo}{\'E}.~Borel. \newblock  Les probabilit{\'e}s d{\'e}nombrables et leurs applications  arithm{\'e}tiques. \newblock {\em Rend. Circ. Mat. Palermo}, 27(1):247--271, 1909.



\bibitem{BG}
Brattka, V. and Gherardi, G., Weihrauch degrees, omniscience principles and weak computability. The Journal of Symbolic Logic, 76(1):143–176, 2011.



\bibitem{BGH} V. Brattka, G. Gherardi, and R. H\"olzl. Probabilistic computability and choice. Information and Computation, 242:249--286, 2015.

\bibitem{BHK} V. Brattka, M. Hendtlass, and A. P. Kreuzer. On the uniform computational content of the Baire category theorem. Notre Dame Journal of Formal Logic, 59(4):605--636, 2018.

\bibitem{Bra} Braverman, M., On the Complexity of Real Functions
\emph{Foundations of Computer Science, 2005, FOCS'05}.

\bibitem{Br1} M. de Brecht. Topological and Algebraic Aspects of Algorithmic Learning Theory. PhD thesis, Graduate School of Informatics, Kyoto University, 2010.

\bibitem{Br2} M. de Brecht and A. Yamamoto. Mind change complexity of inferring unbounded unions of restricted pattern languages from positive data. Theoret.~Comput.~Sci., 411(7-9):976--985, 2010.

\bibitem{Br3} M. de Brecht and A. Yamamoto. Topological properties of concept spaces (full version). Inform. and Comput., 208(4):327--340, 2010.


\bibitem{Ch}A.~Church. \newblock On the concept of a random sequence. \newblock {\em Bull. Am. Math. Soc.}, 46(2):130--135, 1940.


\bibitem{ceremdo}Douglas Cenzer, Rodney~G. Downey, Jeffrey~B. Remmel, and Zia Uddin. \newblock Space complexity of abelian groups. \newblock {\em Arch. Math. Log.}, 48(1):115--140, 2009.

\bibitem{cerem} Douglas~A. Cenzer and Jeffrey~B. Remmel. \newblock Polynomial-time versus recursive models. \newblock {\em Ann. Pure Appl. Logic}, 54(1):17--58, 1991.




\bibitem{ceremab}Douglas~A. Cenzer and Jeffrey~B. Remmel. \newblock Polynomial-time abelian groups. \newblock {\em Ann. Pure Appl. Logic}, 56(1-3):313--363, 1992.

\bibitem{CR-survey98}
D.~Cenzer and J.~B. Remmel.
\newblock Complexity theoretic model theory and algebra.
\newblock In {Yu.}~L. Ershov, S.~S. Goncharov, A.~Nerode, and J.~B. Remmel,
  editors, {\em Handbook of recursive mathematics, {V}ol.\ 1}, volume 138 of
  {\em Stud. Logic Found. Math.}, pages 381--513. North-Holland, Amsterdam,
  1998.

\bibitem{CTo}
Chandra T. D. and Toueg, S. Unreliable failure detectors for reliable distributed systems. JACM 43(2):225–267, 1996.

\bibitem{Cleave} J. Cleave. The primitive recursive analysis of ordinary differential equations and the complexity of their solutions. Journal of Computer and Systems Sciences, 3:447--455, 1969.

\bibitem{C}
Csima, B., \emph{Applications of Computability Theory to Prime Models and Differential Geometry}, Ph.D. dissertation, The University of Chicago, 2003.

\bibitem{Cygea} Cygan, Marek; Fomin, Fedor V.; Kowalik, Lukasz; Lokshtanov, Daniel; Marx, Daniel; Pilipczuk, Marcin; Pilipczuk, Michal; Saurabh, Saket (2015). Parameterized Algorithms. Springer. p. 555.


\bibitem{De} M.~Dehn. \newblock \"{U}ber unendliche diskontinuierliche {G}ruppen. \newblock {\em Math. Ann.}, 71(1):116--144, 1911.

\bibitem{DoT}R.~Downey. \newblock Turing and randomness. \newblock In B.~J. Copeland, J.~P. Bowen, M.~Sprevak, and R.~Wilson, editors,  {\em The {T}uring guide}, pages 427--436. Oxford University Press, Oxford,  2017.

\bibitem{DF} Downey, R. and Fellows, M.
\emph{Fundamentals of Parameterized Complexity},
Springer-Verlag, 2013.


\bibitem{DH} Downey, R., and Hirschfeldt, D.,
\emph{Algorithmic Randomness and Complexity}, Springer-Verlag, 2010.

\bibitem{Dea}
Downey, R., Harrison-Trainor, M.,   Kalimullin, I.,   Melnikov, A.,
and Turetsky, D.,  Graphs are not universal for online computability,
\emph{Journal of Computing and System Sciences}.
Volume 112, September 2020, Pages 1-12.

\bibitem{DHL}
Downey, R., LaForte, G., and Hirschfeldt, D.,
Randomness  and reducibility,
 extended abstract
appeared in  {\em Mathematical Foundations of Computer Science, 2001}
J. Sgall, A. Pultr, and P. Kolman (eds.), Mathematical Foundations of
Computer Science 2001, Lecture Notes in Computer Science 2136 (Springer,
2001), 316--327).
final version in {\em Journal of Computing and System Sciences}.
Vol. 68 (2004), 96-114.

\bibitem{DM}
R. Downey and C. McCartin,
Online problems, pathwidth, and persistence,
{\em Proceedings IWPEC 2004.}
Springer-Verlag
Lecture Notes in Computer
Science 3162, pp 13-24, 2004.



\bibitem{DM1} R. Downey and C. McCartin,
Online promise problems with online
width metrics,
{\em Journal of Computing and System Sciences},
Vol. 73, No. 1 (2007), 57-72.


\bibitem{WordProcessing}
David B.~A. Epstein, James~W. Cannon, Derek~F. Holt, Silvio V.~F. Levy,
  Michael~S. Paterson, and William~P. Thurston.
\newblock {\em Word processing in groups}.
\newblock Jones and Bartlett Publishers, Boston, MA, 1992.


\bibitem{ErGon}Y.~Ershov and S.~Goncharov. \newblock {\em Constructive models}.
\newblock  Siberian School of Algebra and Logic. Consultants Bureau, New York,  2000.


\bibitem{FW} Fiat, A., and Woeginger, G.,
\emph{Online Algorithms}, Springer LNCS, vol 1442.

\bibitem{Fi}
Fishburn, P., Interval graphs and interval orders,
\emph{Discrete Mathematics},
Volume 55, Issue 2, July 1985, Pages 135-149.

\bibitem{FG} Flum, J., and Grohe, M.,
\emph{Parameterized Complexity Theory},
Springer-Verlag, 2006.

\bibitem{FK}
 Friedman, H. and Ko, K, Computational complexity of real functions. \emph{Theoret. Comput. Sci.,} 20(3):323–352, 1982.

\bibitem{GJ} Garey, R. and Johnson, D.,
\emph{Computers and Intractability},
W. H. Freeman, 1979.

\bibitem{Gas} Gasarch, W., A survey of recursive combinatorics, in  \emph{Handbook of Recursive Mathematics}, Vol. 2.,
Edited by Ershov, Goncharov, Marek, Nerode, and Remmel. 1998. Pages 1041–1176. Published by Elsevier ISBN: 0444544249.

\bibitem{Gr55}
Grzegorczyk, A.,
Computable functionals,
\emph{ Fundamenta mathematica},
Vol. 42 (1955), 168-202.

\bibitem{Grigorieff90}
Serge Grigorieff.
\newblock Every recursive linear ordering has a copy in {$DTIME$}-{$SPACE(n,
  log(n))$}.
\newblock {\em J. Symb. Log.}, 55(1):260--276, 1990.


\bibitem{Gol} Gold, M.,
Language Identification in the Limit, \emph{Information and Control,} Vol. 10 (1967), 447–474.

\bibitem{Good} R. Goodstein. Recursive Analysis. Studies in Logic and the Foundations of Mathematics. North-Holland, Amsterdam, 1961.

\bibitem{Gr}
Grohe, M.,
\emph{Descriptive Complexity, Canonisation, and Definable Graph Structure Theory,}
Lecture Notes in Logic, Volume 47. Cambridge University Press, 2017


\bibitem{Herm1926}Grete Hermann. \newblock Die {F}rage der endlich vielen {S}chritte in der {T}heorie der  {P}olynomideale. \newblock {\em Math. Ann.}, 95(1):736--788, 1926.

\bibitem{HirstMum}
J. L. Hirst and C. Mummert. Reverse mathematics of matroids. In A. Day, M. Fellows, N. Greenberg, B. Khoussainov, A. Melnikov, and F. Rosamond, editors, Computability and Complexity: Essays Dedicated to Rodney G. Downey on the Occasion of His 60th Birthday, volume 10010 of Lecture Notes in Computer Science, pages 143–159. Springer, Cham, 2017.

\bibitem{Hod} Hodgson, Bernard R.
Theories decidables par automate fini.
Thesis (Ph.D.)–Universite de Montreal (Canada). 1977.

\bibitem{Hod1} Hodgson, Bernard R.
On direct products of automaton decidable theories.
Theoret. Comput. Sci. 19 (1982), no. 3, 331–335.

\bibitem{HoJSt} R. H\"olzl, S. Jain, and F. Stephan. Inductive inference and reverse mathematics. Annals of Pure and Applied Logic, 167(12):1242--1266, 2016.

\bibitem{Ir1}
Irani, S., Coloring Inductive Graphs On-Line,
in \emph{Proceedings for for the 31st Symposium on the Foundations of Computer Science, 1990,} pp. 470--479.




\bibitem{Ir2}
Irani, S.,
Coloring Inductive Graphs On-Line,
\emph{Algorithmica}, vol.11, no.1, Jan. 1994, pp.53-72.




\bibitem{KMM}
Kalimullin, I., Melnikov, A., and Montalban, A., Punctuality on a cone. In preparation.

\bibitem{KMN1}
Kalimullin, I., Melnikov, A., and  Ng, KM., The Diversity of Categoricity Without Delay. \emph{Algebra and Logic.} 56(2) (2017), 171-177.

\bibitem{KMN2}
Kalimullin, I., Melnikov, A., and  Ng, KM.,
 Algebraic structures computable without delay. \emph{Theoretical Computer Science.}  674 (2017), 73-98.


\bibitem{Karp}
Karp, R., Reducibility Among Combinatorial Problems  In R. E. Miller; J. W. Thatcher (eds.). \emph{Complexity of Computer Computations}. New York: Plenum.
(1972), 85-103.

\bibitem{Karp2}
Karp, R.,
On-line algorithms versus off-line algorithms: How much is it worth to know the future? IFIP Congress (1). 12 (1992), 416-429.

\bibitem{KC} Kawamura, A.,  and Cook S.,
Complexity Theory for Operators in Analysis,
\emph{ACM Transactions on Computation Theory}, 4(2), Article 5, 2012.

\bibitem{KhRandom} Khoussainov, B.,
A quest for algorithmically random infinite structures. \emph{Proceedings of LICS- CSL 2014 conference.} Vienna, Austria.

\bibitem{KhoussainovNerode95}
Bakhadyr Khoussainov and Anil Nerode.
\newblock Automatic presentations of structures.
\newblock In {\em Logic and Computational Complexity ({I}ndianapolis, {IN},
  1994)}, volume 960 of {\em Lecture Notes in Comput. Sci.}, pages 367--392.
  Springer, Berlin, 1995.

\bibitem{KieTrans}
Kierstead, H.,
An Effective Version of Dilworth's Theorem,
\emph{Trans. Amer. Math. Soc}, 268(1)· November 1981.

\bibitem{hand} Kierstead, H.,
Recursive and On-Line Graph Coloring, In \emph{Handbook of Recursive
Mathematics,} Volume 2, pp 1233-1269, Elsevier, 1998.


\bibitem{KT} Kierstead, H. and Trotter, W.,
An Extremal Problem in Recursive Combinatorics.
\emph{Congressus Numeratium},  33, pp 143-153, 1981.


\bibitem{KQ}
Kierstead, H, and Qin, J.,
Coloring interval graphs with First-Fit.
in \emph{Combinatorics of Ordered Sets,}
 papers from the 4th Oberwolfach Conf., 1991), M. Aigner and R. Wille (eds.), Discrete Math. 144, pp 47-57, 1995.




\bibitem{Ko} Ko, K., Complexity theory of real functions. Progress in Theoretical Computer Science, Birkh\"{a}user Boston, Inc., Boston, MA, 1991.



\bibitem{KW} Kreitz, C., and Weihrauuch, K., Theory of representations,
\emph{Theoretical Computer Science},
Volume 38, 1985, Pages 35-53.

\bibitem{LV} Li, M., and Vitanyi, P.,
\emph{An Introduction to Kolmogorov Complexity and Its Applications.} Texts and Monographs in Computer Science. Springer-Verlag,
1993.
\bibitem{LST}
Lovasz, L.,  Saks, M., and  and Trotter, W.,
An on-line graph coloring algorithm with sublinear performance ratio,
\emph{Discrete Mathematics},
Volume 75, Issues 1–3, May 1989, Pages 319-325.


\bibitem{Mel} Melhorn, K.,
Polynomial and abstract subrecursive classes. \emph{J. Comput. Syst. Sci.,} 12(2):147–
178, 1976.


\bibitem{MNroots}G.~Metakides and A.~Nerode. \newblock The introduction of nonrecursive methods into mathematics. \newblock In {\em The {L}. {E}. {J}. {B}rouwer {C}entenary {S}ymposium  ({N}oordwijkerhout, 1981)}, volume 110 of {\em Stud. Logic Found. Math.},  pages 319--335. North-Holland, Amsterdam, 1982.

\bibitem{Milt} Bro Miltersen, P.,  Subramanian, S.,  Scott Vitter J., Tamassia R.,
Complexity models for incremental computation, \emph{Theoretical Computer Science}, Vol. 130 (1994), 203-236.

\bibitem{Mul} Jean-Michel Muller. Some characterizations of functions computable in on-line arithmetic. IEEE Trans. Com- put., 43(6):752–755, 1994.

\bibitem{NW}
Nabutovsky A., and Weinberger S., The fractal nature of Riem/Diff. I. \emph{Geometrica Dedicata,} 101 (2003), 1-54.

\bibitem{Nies}
Nies, A.,
\emph{Computability and Randomness,} volume 51 of Oxford Logic Guides. Oxford University Press, Oxford, 2009.

\bibitem{PourElRich}
Marian~B. Pour-El and J.~Ian Richards.
\newblock {\em Computability in analysis and physics}.
\newblock Perspectives in Mathematical Logic. Springer-Verlag, Berlin, 1989.

\bibitem{Remmel86}J.~B. Remmel. \newblock Graph colorings and recursively bounded {$\Pi^0_1$}-classes. \newblock {\em Ann. Pure Appl. Logic}, 32:185--194, 1986.

\bibitem{RS86a} Robertson N., and Seymour, P.,
Graph minors II. Algorithmic aspects of tree-width.
\emph{Journal of Algorithms}, Vol.  7, pp 309-322, 1986.

\bibitem{Rogers} Rogers, H.,
Theory of recursive functions and effective computability.
\emph{MIT Press}, xxii+482 pages, 1987.
	
\bibitem{Scho} Schr\"oder, M.,
Extended Admissibility. \emph{Theoretical Computer Science} Vol. 284,  519-538 (2002).

\bibitem{TS}
Sleator, D. and  Tarjan, R., Amortized efficiency of list update and paging rules, \emph{Communications of the ACM,} 28 (2) (1985),  202–208.

\bibitem{Sch} M. Schr\"oder. Fast online multiplication of real numbers. In R. Reischuk and M. Morvan, editors, STACS 97, volume 1200 of Lecture Notes in Computer Science, pages 81--92, Berlin, 1997. Springer. 14th Annual Symposium on Theoretical Aspects of Computer Science, Lubeck, Germany, February 27-- March 1, 1997.

\bibitem{Sh} Shalev-Shwartz, S.. Online Learning: Theory, Algorithms, and Applications. PhD thesis.  Hebrew University, 2007.


\bibitem{So}
Soare,  R., Computability theory and differential geometry, \emph{ The Bulletin of Symbolic Logic,} 10 (2004), 457–486.

\bibitem{Sp1} E. Specker. Nicht konstruktiv beweisbare Satze der Analysis. The Journal of Symbolic Logic, 14(3):145–158, 1949.

\bibitem{Sp2} E. Specker. Der Satz vom Maximum in der Rekursiven Analysis. In A. Heyting, editor, Constructivity in mathematics, Studies in Logic and the Foundations of Mathematics, pages 254–265, Amsterdam, 1959. North-Holland. Proc. Colloq., Amsterdam, Aug. 26–31, 1957.

\bibitem{Tri}K. S. Trivedi and and M. D.Ercegovac. On-line algorithms for division and multiplication.IEEE Trans. Comput., C-26(7):681--687, 1977.

\bibitem{tsankov}
T.~Tsankov.
\newblock The additive group of the rationals does not have an automatic
  presentation.
\newblock {\em J. Symbolic Logic}, 76(4):1341--1351, 2011.

\bibitem{Tur}
Turing, A.
On Computable Numbers, with an Application to the Entscheidungsproblem,
\emph{Proceedings of the London Mathematical Society, Series 2,}
Vol.  42 (1936),  230–265.  Errata appeared in Series 2, 43 (1937),  544–546.

\bibitem{vM} R.~von Mises. \newblock Grundlagen der {W}ahrscheinlichkeitsrechnung. \newblock {\em Math. Z.}, 5(1--2):52--99, 1919.

\bibitem{Wad} Wadge, W., ``Reducibility and determinateness on the Baire space". PhD thesis. Univ. of California, Berkeley, 1983.

\bibitem{Wei91} K. Weihrauch. On the complexity of online computations of real functions. \emph{Journal of Complexity,} 7:380–394, 1991.

\bibitem{Wei00} Weihrauch, K.,
\emph{Computable Analysis}, Springer-Verlag, 2000.




\end{thebibliography}
\end{document}